\documentclass[11pt]{article}
\usepackage{amsmath, amssymb, amsthm}
\usepackage{verbatim}
\usepackage{multicol}
\usepackage{enumerate}
\usepackage{comment}
\usepackage{dsfont}
\usepackage[none]{hyphenat}
\usepackage[unicode]{hyperref}
\hypersetup{
	colorlinks=true,
	linkcolor=blue,
	filecolor=magenta,      
	urlcolor=cyan,
	citecolor=blue
}
\usepackage{pgf}
\usepackage{tikz}
\usetikzlibrary{positioning,arrows,shapes,decorations.markings,decorations.pathreplacing,matrix,patterns}
\tikzstyle{vertex}=[circle,draw=black,fill=black,inner sep=0,minimum size=3pt,text=white,font=\footnotesize]
\usepackage{cleveref}

\date{}
\title{\vspace{-1.2cm} Positive discrepancy, MaxCut, and eigenvalues of graphs}

\author{Eero R\"aty\thanks{Ume\r{a} University, \emph{e-mail}: \textbf{\{eero.raty,istvan.tomon\}@umu.se}. ER is supported by a postdoctoral grant 213-0204 from the Olle Engkvist Foundation.}, Benny Sudakov\thanks{ETH Zurich, \emph{e-mail}: \textbf{benjamin.sudakov@math.ethz.ch}.
Research supported in part by SNSF grant 200021\_196965.}, 
	Istv\'an Tomon\footnotemark[1]
}

\theoremstyle{plain}
\newtheorem{theorem}{Theorem}[section]

\newtheorem{corollary}[theorem]{Corollary}
\newtheorem{claim}[theorem]{Claim}
\newtheorem{lemma}[theorem]{Lemma}
\newtheorem{conjecture}[theorem]{Conjecture}

\newtheorem{prop}[theorem]{Proposition}
\Crefname{theorem}{Theorem}{Theorems}
\Crefname{definition}{Definition}{Definitions}
\Crefname{corollary}{Corollary}{Corollaries}
\Crefname{claim}{Claim}{Claims}
\Crefname{lemma}{Lemma}{Lemmas}
\Crefname{conjecture}{Conjecture}{Conjectures}
\Crefname{problem}{Problem}{Problems}
\Crefname{prop}{Proposition}{Propositions}

\theoremstyle{definition}

\DeclareMathOperator{\disc}{disc}
\DeclareMathOperator{\surplus}{sp}
\DeclareMathOperator{\dfc}{dfc}
\DeclareMathOperator{\pdisc}{pdisc}
\DeclareMathOperator{\Sym}{Sym}

\begin{document}
	
	\maketitle
	\sloppy
	
	\begin{abstract}
		The \emph{positive discrepancy} of a graph $G$ of edge density $p=e(G)/\binom{v(G)}{2}$ is defined as 
		$$\disc^{+}(G)=\max_{U\subset V(G)}e(G[U])-p\binom{|U|}{2}.$$
		In 1993, Alon proved (using the equivalent terminology of \emph{minimum bisections}) that if $G$ is $d$-regular on $n$ vertices, and $d=O(n^{1/9})$, then $\disc^{+}(G)=\Omega(d^{1/2}n)$. We greatly extend this by showing that if $G$ has \emph{average degree} $d$, then
		$$\disc^{+}(G)=\begin{cases} \Omega(d^{\frac{1}{2}}n) &\mbox{if }d\in [0,n^{\frac{2}{3}}],\\
			\Omega(n^2/d) & \mbox{if } d\in  [n^{\frac{2}{3}},n^{\frac{4}{5}}],\\
			\Omega(d^{\frac{1}{4}}n/\log n) & \mbox{if } d\in \left[n^{\frac{4}{5}},(\frac{1}{2}-\varepsilon)n\right].\end{cases}$$
		These bounds are best possible if $d\ll n^{3/4}$, and the complete bipartite graph shows that $\disc^{+}(G)=\Omega(n)$ cannot be improved if $d\approx n/2$. Our proofs are based on semidefinite programming and linear algebraic techniques.

        An interesting corollary of our results is that every $d$-regular graph on $n$ vertices with ${\frac{1}{2}+\varepsilon\leq \frac{d}{n}\leq 1-\varepsilon}$ has a \emph{cut} of size $\frac{nd}{4}+\Omega(n^{5/4}/\log n)$. This is not necessarily true without the assumption of regularity, or the bounds on $d$. 
  
		The positive discrepancy of regular graphs is controlled by the second eigenvalue $\lambda_2$, as $\disc^{+}(G)\leq \frac{\lambda_2}{2} n+d$. As a byproduct of our arguments, we present lower bounds on $\lambda_2$ for regular graphs, extending the celebrated Alon-Boppana theorem in the dense regime. 
	\end{abstract}


	
	\section{Introduction}
	Discrepancy theory studies the extent of deviation between expected patterns and actual arrangements of objects in various mathematical structures, with applications spanning measure theory, combinatorics, number theory, and computer science. For a detailed overview of the subject, we refer the interested reader to the book of Chazelle \cite{disc_book}.
	
	Given a graph $G$ with $n$ vertices, $m$ edges, and density $p=m/\binom{n}{2}$, the discrepancy of $G$ is defined as
	$$\disc(G)=\max_{U\subset V(G)}\left|e(G[U])-p\binom{|U|}{2}\right|.$$
	In other words, $\disc(G)$ measures how much the induced subgraphs of $G$ deviate from their expected size.  In 1971, Erd\H{o}s and Spencer \cite{ES71} proved that if $p=1/2$, then the discrepancy is always at least $\Omega(n^{3/2})$. This can easily seen to be tight by considering a random graph of edge density $1/2$. Later, Erd\H{o}s, Goldberg, Pach, and Spencer \cite{EGPS} extended this by showing that $\disc(G)\geq \Omega(\sqrt{mn})$ as long as $p\leq 1/2$. As the discrepancy of $G$ is the same as the discrepancy of its complement, this implies a lower bound for all values of $p$, and it is also tight for $p\in [1/n,1-1/n]$   by considering the Erd\H{o}s-R\'enyi random graph $G(n,m)$. 
	
	The previous results tell us that when $p\leq 1/2$, we can find some subset $U\subset V(G)$ such that $e(G[U])$ is either $\Omega(\sqrt{mn})$ larger than its expected size $p\binom{|U|}{2}$, or it is  $\Omega(\sqrt{mn})$ smaller. This raises the following natural question: to what extent can we expect both outcomes to hold? That is, how large of a surplus and how large of a deficit are we guaranteed to find in an induced subgraph of $G$ compared to its expected size? This motivates the following definitions. The \emph{positive discrepancy} of $G$ is defined as
	$$\disc^{+}(G)=\max_{U\subset V(G)}e(G[U])-p\binom{|U|}{2},$$
	and the \emph{negative discrepancy} of $G$ is defined as 
	$$\disc^{-}(G)=\max_{U\subset V(G)}p\binom{|U|}{2}-e(G[U]).$$
	Thus, we have $\disc(G)=\max\{\disc^{+}(G),\disc^{-}(G)\}$. These two quantities are closely related to some of the most extensively studied parameters of graphs.
	
	The MaxCut problem studies the maximum size of a cut, i.e.\ the maximum number of edges between two parts in a partition of the vertex set. MaxCut is a central problem both in discrete mathematics and theoretical computer science \cite{AlonMaxCut,AKS05,BJS,Edwards1,Edwards2,GJS,GW95}. A simple probabilistic argument shows that the maximum cut is always at least $m/2$, so it is more natural to study the \emph{surplus}, defined as follows. The \emph{surplus} of $G$, denoted by $\surplus(G)$, is the maximum $s$ such that $V(G)$ has a partition into two parts with $m/2+s$ edges between the parts. As we show in Lemma \ref{lemma:surplus}, in case $G$ is regular, we have $\surplus(G)=\Theta(\disc^{-}(G))$. One can define the \emph{minimum bisection} and \emph{deficit} of a graph in a similar manner (but by restricting partition to equipartition): let $\dfc(G)$ denote the maximum $s$  such that $G$ has a partition into sets of size $\lfloor n/2\rfloor$ and $\lceil n/2\rceil$ with $m/2-s$ edges between the parts; then $m/2-s$ is the size of the minimum bisection. Again, in case $G$ is regular, we show  $\dfc(G)=\Theta(\disc^{+}(G))$.
	
	Now let us focus on the positive discrepancy. In 1993, Alon \cite{Alon93} proved that in case $G$ is $d$-regular and $d=O(n^{1/9})$, we have $\dfc(G)=\Omega(d^{1/2}n)=\Omega(\sqrt{mn})$, and equivalently $\disc^{+}(G)=\Omega(d^{1/2}n)$; later, similar results were obtained for multigraphs by Alon, Hamburger, and Kostochka \cite{AHK99}. These results are tight as a random $d$-regular graph for $1\leq d\leq n/2$ has both positive and negative discrepancy equal to $\Theta(d^{1/2}n)$ with high probability. Alon's argument is based on combinatorial and probabilistic ideas, which require the graph to be sufficiently sparse. On the other hand, the complete bipartite graph with parts of size $n/2$ has positive discrepancy $O(n)$, so the lower bound  $\Omega(d^{1/2}n)$ does not need to hold in general.
 
 Bollob\'as and Scott \cite{BS06} establishes the lower bound $\disc^{+}(G)=\Omega(n)$ for every $G$ of average degree $d\in [1,n-1]$. However, this bound is only known to be tight when $d\approx n(1-1/r)$ for some integer constant $r$, the construction being the Tur\'an graph $T_r(n)$. We recall that $T_{r}(n)$ is the union of $r$ independent sets, each of size $\lfloor n/r\rfloor$ or $\lceil n/r\rceil$, with all edges added between the sets. The lower bound $\disc^{+}(G)=\Omega(n)$ actually follows from a more general theorem, which establishes the following curious interplay between the positive and negative discrepancy: if  $d\in [1,n-1]$, then
	$$\disc^{+}(G)\cdot \disc^{-}(G)=\Omega(d(n-d)n).$$

 \subsection{Our results}
	
	What happens when $d$ is between $n^{1/9}$ and $n/2$? The following conjecture is proposed by Verstraete (see Conjecture A in \cite{V17}).
 
 \begin{conjecture}
     Let $G$ be an $n$ vertex graph of average degree $d$, where $1\leq d\leq (1/2-\varepsilon)n$. Then $\disc^+(G)=\Omega(d^{1/2}n)$.
 \end{conjecture}
  
 As we will show, this conjecture is false in general and the behavior of the minimum positive  discrepancy turns out to be more complex.  To this end, let $f(d)=f_n(d)$ denote the minimum of $\disc^{+}(G)$ over all graphs on $n$ vertices of average degree $d$ (naturally, we consider only those values of $d$ such that $dn$ is an even integer).  We show that $f(d)=\Theta(d^{1/2}n)$ continues to hold as long as $d\ll n^{2/3}$, after which $f(d)$ starts decreasing until $d\approx n^{3/4}$. For $n^{2/3}\leq d\leq n-1$, we prove the lower bound $\Omega(n^2/d)$, which is the best possible if $d\in [n^{2/3}, n^{3/4}]$, see  Section \ref{sect:upper_bound} for a discussion about upper bounds. Starting from $d\gg n^{3/4}$, we do not have matching lower and upper bounds anymore. However, we improve $\Omega(n^2/d)$  when $n^{4/5}\ll d\leq (1/2-\varepsilon)n$ by establishing the stronger lower bound $f(d)=\Omega(d^{1/4}n/\log n)$.  Finally, the complete balanced bipartite graph shows that $f(n/2)=O(n)$. Thus, strangely, it seems like the function $f(d)$ changes monotonicity at least 3 times between $1$ and $n/2$, see Figure \ref{figure1} for an illustration.  The lower bounds are summarized in the following theorem.
	
	\begin{theorem}\label{thm:main}
		There exists $c>0$, and for every $\varepsilon\in (0,1/2)$ there exists $c_1>0$ such that  the following holds for every sufficiently large $n$. Let $G$ be graph on $n$ vertices of average degree $d$. Then
		$$\disc^{+}(G)> \begin{cases} cd^{1/2}n &\mbox{if }d\leq n^{2/3},\\
			c n^2/d & \mbox{if }n^{2/3}\leq d\leq n^{4/5},\\
			c_1d^{1/4}n/\log n & \mbox{if }n^{4/5}\leq d\leq (1/2-\varepsilon)n.\end{cases}$$
	\end{theorem}

\noindent
We want to emphasize that an immediate corollary of this theorem, greatly extends the main result of Alon \cite{Alon93} about the minimum bisection of $d$ regular $n$-vertex graphs $G$.
It shows that $\dfc(G) >cd^{1/2}n$ even when $d$ is as large as $n^{2/3}$, and that $\dfc(G) >cn^2/d$ for $n^{2/3}\leq d\leq n^{4/5}$. As we will show in the next section both of these bounds are the best possible if $d\leq n^{3/4}$.
Finally for $n^{4/5}\leq d\leq (1/2-\varepsilon)n$, the deficit is always at least $c d^{1/4}n/\log n$.
 
    The proof of Theorem \ref{thm:main} is based on semidefinite programming and linear algebraic techniques. The proof of the first two inequalities is quite short, and is partially motivated by recent results on the MaxCut problem \cite{BJS,Max-Cut-SDP,GJS}, see Section \ref{sect:sparse}. The proof of the third inequality, however, is much more involved. First, we  show that $\disc^{+}(G)$ is close to the solution of a semidefinite optimization program, see Section \ref{sect:semidefinite}. 
    Then, we present a number of lower bounds on this problem based on the eigenvalue decomposition of the adjacency matrix of $G$, see Section \ref{sect:easy}. Finally, Section \ref{sect:hard} is the most involved, but also the most fascinating part of our paper, where we prove the third inequality. While the proof of  the first two inequalities may have a combinatorial shadow as well, the proof of the third inequality is heavily based on linear algebraic and geometrical arguments.

	\subsection{Upper bounds}\label{sect:upper_bound}
	
	In this section, we discuss the tightness of Theorem \ref{thm:main}. As we mentioned earlier, the random $d$-regular graph satisfies $\disc^{+}(G)=\Theta(d^{1/2}n)$ for $d\leq n/2$, so let us consider $d\gg n^{2/3}$.
	
	It the case of regular graphs, the positive discrepancy is controlled by $\lambda_2$, the second largest eigenvalue of the adjacency matrix of $G$, while the negative discrepancy is controlled by $\lambda_n$, the smallest eigenvalue. In particular, $\disc^{+}(G)\leq \frac{\lambda_2}{2}n+d$, see Lemma \ref{lemma:lambda_2_disc}. Therefore, in order to find a graph with small positive discrepancy, it is enough to look for graphs with small $\lambda_2$.  There are several good candidates among \emph{strongly-regular} and \emph{distance-regular} graphs, see \cite{BCN89} for a  detailed overview of the topic, and \cite{DKT14} for a more recent survey.
	
	A graph $G$ is \emph{strongly-regular} of parameters $(n,d,r,s)$ if $G$ is $d$-regular on $n$ vertices, any two adjacent vertices have exactly $r$ common neighbors, and any two non-adjacent vertices have $s$ common neighbors. It is well known (see e.g. Proposition 2.12 in \cite{KS06}) that the second eigenvalue of $G$ can be computed as 
     $$\lambda_2= \frac{1}{2}\left(r-s+\sqrt{(r-s)^2+4(d-s)}\right).$$
 A construction, due to R. Metz (see \cite{LB84}, Section 7.A), considers a non-singular elliptic quadric in the projective space $PG(5,q)$ to define a strongly-regular graph $Q=Q(q)$ with parameters 
	\begin{align*}
		n&=\frac{1}{2}q^2(q^2-1)\\
		d&=(q-1)(q^2+1)\\
		r&=(q-1)(q+2)\\
		s&=2q(q-1).
	\end{align*} 
	The second largest eigenvalue of this graph is $\lambda_2=q-1$, while the smallest eigenvalue is $\lambda_n=-(q-1)^2$. Therefore, these provide an infinite family of regular graphs with $d=\Theta(n^{3/4})$ and $\lambda_2=\Theta(n^{1/4})=\Theta(n/d)$, and thus having positive discrepancy $O(n^2/d)$. We note that there are further examples with similar parameters, see e.g. \cite{Brouwer}.  We can use these graphs to construct for every $d\in [n^{2/3},n^{3/4}]$ a graph of average degree $\approx d$ and positive discrepancy $O(n^2/d)$.
	
	\begin{prop}\label{prop:upper_bound}
		Let $\varepsilon\in (0,1/2)$, then there exists $c$ such that the following holds for every sufficiently large $n$. Let $n^{2/3}\leq d\leq n^{3/4}$, then there exists a graph on $n$ vertices of average degree between $d(1-\varepsilon)$ and $d(1+\varepsilon)$ such that $\disc^{+}(G)\leq cn^2/d$.
	\end{prop}
        Recently, Davis, Huczynska, Johnson, and Polhill \cite{DHJP} constructed a family of strongly-regular graphs based on partial difference sets with the following  parameters. For every odd prime $p$ and integer $m\geq 2$, the graph $D(p,m)$ satisfies
        \begin{align*}
		n&=p^{3m}\\
		d&=(p^{2m-1}-p^m+p^{m-1})(p^m-1)\\
		r&=p^m-p^{m-1}+(p^{2m-1}-p^{m}+p^{m-1})(p^{m-1}-2)\\
		s&=(p^{2m-1}-p^m+p^{m-1})(p^{m-1}-1).
	\end{align*}
        This graph has second eigenvalue $\lambda_2=p^{m}-p^{m-1}=\Theta(n^{1/3})$. Note that when $p$ is a constant (e.g. $p=3$), this gives an infinite family of graphs with $d=(1+o(1))\frac{n}{p}=\Omega(n)$ and positive discrepancy $O(n^{4/3})$. Also, taking $m$ to be a constant instead, we get family with $d=\Theta(n^{1-1/3m})$ and $\lambda_2=\Theta(n^{1/3})=\Theta(d^{1/(3-1/m)})$. 
        
        We believe that in the range $n^{3/4}\ll d\leq (1/2-\varepsilon)n$, the smallest possible value of the discrepancy is $\Theta(d^{1/3}n)$. We know that this is indeed the case for $d\approx n^{3/4}$, and there is a matching upper bound for $d=\Theta(n)$. See also the discussion about the eigenvalues (Section \ref{sect:intr_eigen}) for even more evidence.
	
	\begin{conjecture}
		Let $n^{3/4}<d\leq n(1/2-\varepsilon)$, and let $G$ be a $d$-regular graph on $n$ vertices. Then 
		$$\disc^{+}(G)=\Omega_{\varepsilon}(d^{1/3}n).$$
	\end{conjecture} 

 \subsection{MaxCut in very dense graphs}

 Theorem \ref{thm:main} has an interesting corollary about the maximum cut. A classical result of Edwards \cite{Edwards1,Edwards2} states that  every graph $G$ with $n$ vertices and $m$ edges satisfies $\surplus(G)\geq \frac{\sqrt{8m+1}-1}{8}$ and this is tight when $G$ is the complete graph on an odd number of vertices. In general, if $G$ is the disjoint union of a constant number of cliques, then $\surplus(G)=O(\sqrt{m})$. Then, the question naturally arises whether this bound can be improved if $G$ is far from such a configuration. One way to ensure this is to assume that $G$ avoids some graph $H$ as a subgraph. The study of MaxCut in $H$-free graphs was initiated by Erd\H{o}s and Lov\'asz (see \cite{erdos}) in the 70's, and substantial amount of research was devoted to this problem since, see e.g. \cite{AlonMaxCut,AKS05,BJS,GJS,GW95}. 
 
 Here, we study a different setup. We consider regular graphs with degree between $n(1/2+\varepsilon)$ and $n(1-\varepsilon)$. This is a large family of graphs that are also automatically far from the unions of cliques. In this case, we get the following significant improvement over the Edwards bound.

 \begin{theorem}\label{thm:maxcut}
     For every $\varepsilon>0$ there exists $c>0$ such that the following holds. Let $G$ be a $d$-regular graph on $n$ vertices, where $d\in [(\frac{1}{2}+\varepsilon)n,(1-\varepsilon)n]$. Then $$\surplus(G)\geq \frac{cn^{5/4}}{\log n}.$$
 \end{theorem}

 This theorem immediately follows from the third inequality in Theorem \ref{thm:main}, by noting that $\Theta(\surplus(G))=\disc^{-}(G)=\disc^{+}(\overline{G})$; see Lemma \ref{lemma:surplus} for the first equality. We point out that by taking the complement of the graph $D(3,m)$ defined in Section \ref{sect:upper_bound}, we get that the lower bound in Theorem \ref{thm:maxcut} cannot be improved beyond $\Theta(n^{4/3})$.
	
	\subsection{The second eigenvalue}\label{sect:intr_eigen}
	
	The Alon-Boppana theorem \cite{alon-boppana} is one of the cornerstone results of spectral graph theory. It states that if $G$ is an $n$-vertex  $d$-regular graph of diameter $D$, then the second largest eigenvalue $\lambda_2$ of the adjacency matrix is at least $$2\sqrt{d-1}-\frac{2\sqrt{d-1}-1}{\lfloor D/2\rfloor}.$$
	In particular, if $D\rightarrow \infty$ (which is satisfied in case $d=n^{o(1)}$), one has $\lambda_2\geq 2\sqrt{d-1}-o(1)$. So called Ramanujan graphs show that the latter bound is tight \cite{ramanujan}.   Unfortunately, in case $D\in \{2,3\}$, which can already happen if $d\approx n^{1/3}$, the Alon-Boppana bound does not tell us much. Note that the examples from Section \ref{sect:upper_bound} also show that we cannot expect $\lambda_2=\Omega(d^{1/2})$ to hold in general, and in case $d=n/2$, the complete bipartite graphs achieves $\lambda_2=0$. Somewhat surprisingly, there are not many results that address lower bounds on $\lambda_2$ in case $n^{1/3}\ll d<n/2$. Recently, Balla \cite{Balla21} proved that $\lambda_2=\Omega(d^{1/3})$ if $d\leq (1/2-\varepsilon)n$, using fairly involved geometric arguments and matrix projections. We point out that this bound also follows from a very short argument of Ihringer \cite{Ihringer} after a bit of work. Indeed, \cite{Ihringer}, Proposition 1.3 implies that $\lambda_2=\Omega(\sqrt{|\lambda_n|})$, where $\lambda_n$ is the smallest eigenvalue. However, the inequality $(1+\lambda_2)\cdot|\lambda_n|\geq \Omega(d)$ is easy to establish (see Proposition \ref{prop:product}), implying the desired lower bound $\lambda_2=\Omega(d^{1/3})$.
	
	As discussed in Section \ref{sect:upper_bound}, we have $\frac{\lambda_2}{2}n+d\geq \disc^{+}(G)$ (Lemma \ref{lemma:lambda_2_disc}), so lower bounds on $\disc^{+}(G)$ imply lower bounds on $\lambda_2$. In particular, Theorem \ref{thm:main} implies that $\lambda_2=\Omega(d^{1/2})$ if $d\leq n^{2/3}$, $\lambda_2 = \Omega(n/d)$ for $n^{2/3}\leq d \leq n^{4/5}$. These bounds are also implicit in a recent work of Balla \cite{Balla21}. For the interval $n^{4/5}\leq d \leq (1/2-\varepsilon)n$, the discrepancy lower bound implies that $\lambda_2 = \Omega(d^{1/4}/\log n)$. In the next theorem, we improve this bound in the range  $n^{3/4}\leq d\leq (1/2-\varepsilon)n$ to $\Omega(d^{1/3})$. Our proof has some similarity to that of Balla \cite{Balla21} and Ihringer \cite{Ihringer}, but we also relate $\lambda_2$ to the number of triangles, see Theorem \ref{thm:lambda_2_and_delta}.
	
	\begin{theorem}\label{thm:eigenvalues}
		There exists $c>0$, and for every $\varepsilon>0$ there exists $c_1>0$ such that  the following holds for every sufficiently large $n$. Let $G$ be a $d$-regular graph on $n$ vertices, and let $\lambda_2$ denote the second largest eigenvalue of the adjacency matrix of $G$. Then
		$$\lambda_2>\begin{cases} cd^{1/2} &\mbox{if }d\leq n^{2/3},\\
			c n/d & \mbox{if }n^{2/3}\leq d\leq n^{3/4},\\
			c_1d^{1/3} & \mbox{if }n^{3/4}\leq d\leq (1/2-\varepsilon)n.\end{cases}$$
	\end{theorem} 
	
	The random $d$-regular graph satisfies $\lambda_2=\Theta(d^{1/2})$ for every $d$, so Theorem \ref{thm:eigenvalues} is sharp for $d\leq n^{2/3}$. The construction $Q=Q(q)$ described in Section \ref{sect:upper_bound} shows that it is also sharp at $d\approx n^{3/4}$. However, by relaxing regularity, it is also sharp for $n^{2/3}\leq d\leq n^{3/4}$. Indeed, fix some  $q$, and let $d(Q)$ denote the regularity of $Q$, and $\lambda_2(Q)$ denote the second eigenvalue of $Q$; recall that $\lambda_2(Q)=O(v(Q)/d(Q))$. We can choose $q$ such that $\frac{n}{d}=(1+o(1)) \frac{v(Q)}{d(Q)}$, then $v(Q)\geq n$. A random $n$ vertex induced subgraph of $Q$ then has all degrees $d(1+o(1))$ with high probability, and by the Cauchy interlacing theorem, 
	$$\lambda_2\leq \lambda_2(Q)\ll \frac{v(Q)}{d(Q)}\approx \frac{n}{d}.$$
 Moreover, by the strongly regular graph  $D(q,m)$ discussed in Section \ref{sect:upper_bound}, the bound $\lambda_2=\Omega(n^{1/3})$ is sharp when $d=\Omega(n)$. It remains open what the smallest possible value of $\lambda_2$ is when $n^{3/4}\ll d\ll n^{1-\varepsilon}$, but it might be reasonable to conjecture that the lower bound provided by Theorem \ref{thm:eigenvalues} gives the right answer. See Figure \ref{figure1} for an illustration of the relation between our lower bounds on $\disc^{+}(G)$ and $\lambda_2$, and the upper bounds.
	
	\bigskip
	\begin{figure}
		\begin{center}
			
			\begin{tikzpicture}
				\draw[->] (0,0) -- (0,5);
				\draw[->] (0,0) -- (10,0);
				\draw[fill] (8,0) circle (0.1); \node at (8,-1) {$d=n(\frac{1}{2}-\varepsilon)$}; \draw[->] (8,-0.7) -- (8,-0.2);
				\draw[fill] (8.3,0) circle (0.1);
				\draw[fill] (5.328,0) circle (0.1); \node at (5.328,-0.5) {$\frac{2}{3}$};
				\draw[fill] (6,0) circle (0.1); \node at (6,-0.5) {$\frac{3}{4}$};
				\draw[fill] (6.4,0) circle (0.1); \node at (6.4,-0.5) {$\frac{4}{5}$};
				\draw[fill] (0,4) circle (0.1); \node at (-0.5,4) {$\frac{1}{2}$};
				\draw[fill] (0,2.666) circle (0.1); \node at (-0.5,2.666) {$\frac{1}{3}$};
				\draw[dotted] (0,4) -- (8,4) ;
				\draw[dotted] (8,0) -- (8,4) ;
				\draw[dotted] (5.328,0) -- (5.328,2.664) ;
				\draw[dotted] (6,0) -- (6,2) ;
				\draw[dotted] (6.4,0) -- (6.4,1.6) ;
				\draw[dotted] (0,2.666) -- (8,2.666) ;
				\node at (-0.5,5) {$\beta$};
				\node at (10,-0.5) {$\alpha$};
				
				\node at (9.3,0.7) {$d=\frac{n}{2}$}; \draw[->] (9.25,0.5) -- (8.4,0.15);
				
				\draw[line width=1, red] (0,0)  -- (5.328,2.658) -- (6,2) ;
                \draw[line width=1, blue] (0,0.05)  -- (5.328,2.687) -- (6,2.05) ;
                \draw[line width=1, green] (0,0.1)  -- (5.328,2.71) -- (6,2.1) ;
				\draw[rounded corners,line width=1, red] (6,2) -- (6.4,1.6) -- (8,2) -- (8.3,0); 
				\draw[rounded corners,line width=1, blue] (6,2.05) -- (8,2.666) -- (8.3,0) ;
				\draw[rounded corners,line width=1, green] (6.666,2.7) -- (8,2.7) -- (8.3,0);
				
				\node at (10.3,2) {\color{red} $\frac{\disc^{+}(G)}{n}$ (Theorem \ref{thm:main})};
				\node at (10.3,2.666) {\color{blue} $\leq \lambda_2$  (Theorem \ref{thm:eigenvalues})};
				\node at (9,3.2) {\color{green} $\geq \lambda_2$};
			\end{tikzpicture}
			\label{figure1}
			\caption{A summary of our bounds. Writing $d=n^{\alpha}$, the red line denotes our lower bound $n^{1+\beta}=\disc^{+}(G)$ (Theorem \ref{thm:main}), the blue line denotes our lower bound $n^{\beta}=\lambda_2$ (Theorem \ref{thm:eigenvalues}), and the green line denotes the upper bound.}
		\end{center}
	\end{figure}
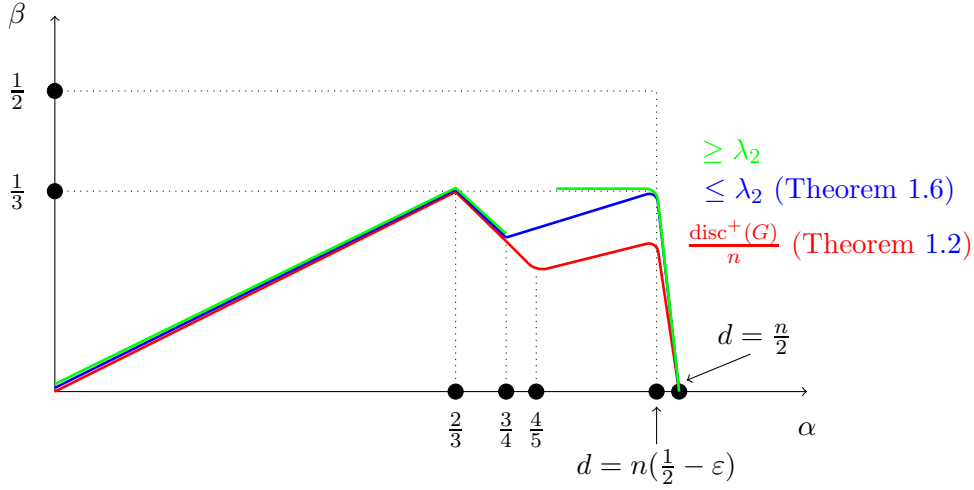
	
	\noindent
	\textbf{Organization.} In the next section, we introduce our notation, state some preliminary results, and establish the connection between discrepancy and surplus. Also, we present the proof of Proposition \ref{prop:upper_bound}. Then, in Section \ref{sect:sparse}, we prove the first two inequalities of Theorem \ref{thm:main}. In what follows, we prove the third inequality in a number of steps. First, in Section \ref{sect:semidefinite}, we show how to relax the positive discrepancy to a semidefinite optimization problem. In Section \ref{sect:easy}, we present lower bounds on the positive discrepancy using the eigenvalues of the adjacency matrix. The following section addresses the third inequality of Theorem \ref{thm:main}, and is the most difficult part of the paper. In Section \ref{sect:eigenvalue}, we prove Theorem \ref{thm:eigenvalues}. We finish our paper with some concluding remarks.  
	
	\section{Preliminaries}
	
	\subsection{Matrix notation}
	Let $\Sym(n)$ denote the set of symmetric, real valued $n\times n$ matrices. In this paper, we only work with symmetric matrices. We use $I=I_n$ to denote the identity matrix, and $J=J_n$ to denote the all-one matrix. For  $A\in\Sym(n)$, we denote by $A_{i,j}$ the corresponding entry of $A$. In case $v\in \mathbb{R}^{n}$ is a vector, we write $v(i)$ for the $i$-th coordinate. Given two matrices $A,B\in\Sym(n)$, their inner product is
	$$\langle A,B\rangle=\sum_{1\leq i,j\leq n}A_{i,j}B_{i,j}.$$
	Then, the \emph{Frobenius norm} of $A$ is defined as 
	$$||A||_{F}^2=\langle A,A\rangle=\sum_{i,j}A_{i,j}^2=\mbox{tr}(A^2).$$
	If $\lambda_1,\dots,\lambda_n$ are the eigenvalues of $A$, we have the identity $||A||_{F}^2=\sum_{i=1}^{n}\lambda_i^2$. We also recall the $\ell_2$-norm of $A$ as 
 $$||A||_2=\max_{v\in\mathbb{R}^n,||v||_2=1}||Av||_2=\lambda_1.$$  
	We denote by $A\succeq 0$ that $A$ is positive semidefinite, and write $A\succeq B$ if $A-B\succeq 0$. A simple observation that we repeateadly use is that if $A\succeq B$, then $A_{i,i}\geq B_{i,i}$ for every $i\in [n]$. 
	
	The \emph{Hadamard product} (or entry-wise product) of $A$ and $B$ is denoted by $A\circ B$, and it is the $n\times n$ matrix whose $(i,j)$ entry is $A_{i,j}\cdot B_{i,j}$. An important property of the Hadamard product, known as the Schur product theorem, is that it preserves positive semidefiniteness.
	
	\begin{lemma}[Schur product theorem]\label{lemma:hadamard}
		If $A,B\succeq 0$, then $A\circ B\succeq 0$. 
	\end{lemma}
	
	Finally, as promised in Section \ref{sect:intr_eigen}, we prove the following inequality between the second and last eigenvalue of a regular graph.
	
	\begin{prop}\label{prop:product}
		Let $G$ be a $d$-regular graph, $d\leq n/2$. Let $\lambda_2$ and $\lambda_n$ be the second largest and smallest eigenvalue of the adjacency matrix, respectively. Then
		$$(1+\lambda_2)\cdot |\lambda_n|\geq \frac{d}{4}.$$
	\end{prop}

	\begin{proof}
		Let $A$ be the adjacency matrix of $G$, and let $d=\lambda_1\geq \lambda_2\geq\dots\geq \lambda_n$ be the eigenvalues of $A$. 
 Since $G$ is not a complete graph its independence number is at least $2$, so by interlacing $\lambda_2\geq 0$. Let $k$ be the largest index such that $\lambda_{k}\geq 0$. As $\mbox{tr}(A)=0$ and $\mbox{tr}(A^2)=nd$, we get  
		$$d+\sum_{i=2}^{k}\lambda_i=\sum_{k+1}^{n}|\lambda_i|,$$
		and
		$$\sum_{i=2}^{n}\lambda_i^2=nd-d^2\geq \frac{nd}{2}.$$
		
		If $\lambda_2\geq \sqrt{d}/2$ and $|\lambda_n|\geq \sqrt{d}/2$, we are done. Hence, we  may assume that $\lambda_2\leq \sqrt{d}/2$, the other case can be handled similarly. Then $\sum_{i=2}^{k}\lambda_i^2\leq \frac{nd}{4}$, so 
		$$\frac{nd}{4}\leq \sum_{i=k+1}^{n}\lambda_i^2\leq |\lambda_n|\sum_{i=k+1}^{n}|\lambda_i|=|\lambda_n|\left(d+\sum_{i=2}^{k}\lambda_i\right)\leq |\lambda_n|(d+n\lambda_2) \leq 
  |\lambda_n|n(1+\lambda_2).$$

	\end{proof}
	
	\subsection{Discrepancy}
	In this section, we collect basic definitions and results about graphs and discrepancy, and prove Proposition \ref{prop:upper_bound}.

	Let $G$ be a graph with $n=v(G)$ vertices, $m=e(G)$ edges, density $p=m/\binom{n}{2}$, and average degree $d=2m/n$. Given $U\subset V(G)$, we write $U^{c}=V(G)\setminus U$, $e(U)$ is the number of edges in the induced subgraph $G[U]$, and 
	$$\disc_G(U)=\disc(U)=e(U)-p\binom{|U|}{2}.$$
	Then, $$\disc^{+}(G)=\max_{U\subset V(G)}\disc(U).$$
 Furthermore, if $U$ and $W$ are disjoint subsets of $V(G)$, write $e(U,W)$ for the number of edges between $U$ and $W$, and let
	$$\disc(U,W)=e(U,W)-p|U||W|.$$
	Observe that $\disc(U\cup W)=\disc(U)+\disc(U,W)+\disc(W)$, and in particular,
	\begin{equation}\label{equ:zero_disc}
		0=\disc(U)+\disc(U,U^c)+\disc(U^c).
	\end{equation}
	
	\begin{claim}\label{claim:complement}
		If $G$ is regular, then $\disc(U)=\disc(U^c)=-\frac{1}{2}\disc(U,U^c)$.
	\end{claim}
	
	\begin{proof}
		It is enough to prove that $\disc(U)=-\frac{1}{2}\disc(U,U^c)$, then $\disc(U)=\disc(U^c)$ follows from (\ref{equ:zero_disc}). But this is true as $p=\frac{d}{n-1}$ and 
		\begin{align*}
			\disc(U,U^c)&=e(U,U^c)-p|U||U^c|=
			|U|d-2e(U)-p|U|(n-|U|)
			=|U|d-2e(U)+p|U|^2-p|U|n\\
			&=|U|d-2\disc(U)+p|U|-p|U|n
			=-2\disc(U).
		\end{align*}
	\end{proof}
	
	The next lemma is a useful technical tool to help us deal with high degree vertices. We will use this to show that either vertices of degree at least $(1+\delta)d$ already cause large positive discrepancy, or we can remove every edge attached to such a high degree vertex without effecting the discrepancies by much.
	
	\begin{lemma}\label{lemma:maxdeg}
		Let $\delta\in (0,1)$, then there exists $c_1,c_2>0$ such that the following holds. Let $X\subset V(G)$ such that the degree of every vertex in $X$ is at least $(1+\delta)d$. Then 
		$$\disc^{+}(G)\geq c_1 (e(X)+e(X,X^c))-c_2 n.$$
	\end{lemma}
	
	\begin{proof}
		In the proof, the $O(.)$ and $\Omega(.)$ notations hide constants which only depend on $\delta$. 
		
		Let $s=|X^c|$, and let $\alpha\in (0,1)$ be specified later, depending only on $\delta$. Furthermore, let $b=\lfloor \alpha s\rfloor$, and let $Y$ be a random $b$ element subset of $X^c$, chosen from the uniform distribution. We have
		$$\disc(X\cup Y)=\disc(X)+\disc(X,Y)+\disc(Y).$$
		Here, $\mathbb{E}(\disc(X,Y))=\frac{b}{s}\disc(X,X^c)=\alpha\disc(X,X^c)+O(n)$, and $\mathbb{E}(\disc(Y))=\frac{b(b-1)}{s(s-1)}\disc(X^c)=\alpha^2\disc(X^c)+O(n)$. We get
		\begin{align*}
			\mathbb{E}(\disc(X\cup Y))&=\disc(X)+\alpha\disc(X,X^c)+\alpha^2\disc(X^c)+O(n)\\
			&=(1-\alpha^2)\disc(X)+(\alpha-\alpha^2)\disc(X,X^c)+O(n),
		\end{align*}
		where we used (\ref{equ:zero_disc}) in the second equality. Here, the right hand side can be  written as
		\begin{align}\label{equ:cc}\begin{split}
				&(1-\alpha^2)e(X)+(\alpha-\alpha^2)e(X,X^c)-p|X|\left(\frac{1-\alpha^2}{2}|X|+(\alpha-\alpha^2)(n-|X|)\right)+O(n)\\
				&\geq (\alpha-\alpha^2)(2e(X)+e(X,X^c)-d|X|)-\frac{p(1-\alpha)^2}{2}|X|^2+O(n).\end{split}
		\end{align}
		
		In the inequality, we used that $1-\alpha^2\geq 2(\alpha-\alpha^2)$ holds for every $\alpha\in [0,1]$. Write $f=2e(X)+e(X,X^c)$, and observe that $f$ is the sum of the degrees of the vertices in $X$. Therefore, $f\geq (1+\delta)d|X|$, which implies that $f-d|X|\geq \frac{\delta}{2}d|X|+\frac{\delta}{4}f$. Hence, the second line of (\ref{equ:cc}) can be lower bounded by
		$$\frac{(\alpha-\alpha^2)\delta}{4}f+(1-\alpha)\left(\frac{\alpha \delta d}{2}|X|-\frac{(1-\alpha)d}{2n}|X|^2\right)+O(n).$$
		Set $\alpha=\frac{1}
		{1+\delta}$, then the second term is nonnegative. Therefore, we get
		$$\mathbb{E}(\disc(X\cup Y))\geq \Omega(f)+O(n).$$
		But then by choosing $Y$ maximizing $\disc(X\cup Y)$, we showed that $\disc^{+}(G)=\Omega(f)+O(n)$. Noting that $f\geq e(X)+e(X,X^c)$, we are done.
	\end{proof}

   Next, we show that the positive discrepancy is indeed controlled by the second eigenvalue in the case of regular graphs.
 
	\begin{lemma}\label{lemma:lambda_2_disc}
		Let $G$ be a $d$-regular graph on $n$ vertices and let $\lambda_2$ be the second largest eigenvalue of the adjacency matrix of $G$. Then for every $U\subset V(G)$, $$\disc(U)\leq  \frac{\lambda_2}{2}|U|+d.$$ In particular, $$\disc^{+}(G)\leq \frac{\lambda_2}{2}n+d.$$
	\end{lemma}
	
	\begin{proof}
		Let $A$ be the adjacency matrix of $G$, $d=\lambda_1\geq\dots\geq\lambda_n$ be the eigenvalues of $A$, and let $v_1=\frac{1}{\sqrt{n}} \mathds{1}, v_2\dots,v_n$ be the corresponding eigenvalues (where $\mathds{1}$ is the all one vector). Also, let $\mathbf{x}$ be the characteristic vector of $U$. Then
  $||x||_2^2=\langle \mathbf{x},\mathds{1}\rangle=|U|$. Recalling that $p=\frac{d}{n-1}$, we have
		\begin{align}
			\disc(U)&=\frac{1}{2}\mathbf{x}^{T}A\mathbf{x}-p\frac{|U|(|U|-1)}{2}=
   \frac{1}{2}\mathbf{x}^{T}A\mathbf{x}-\frac{d}{2(n-1)}(\langle \mathbf{x},\mathds{1}\rangle^2-||x||_2^2)\\
			&=\frac{1}{2}\sum_{i=1}^{n}\lambda_i \langle x,v_i\rangle^2-\frac{d}{2n}\langle \mathbf{x},\mathds{1}\rangle^2-\frac{d}{2n(n-1)}\langle \mathbf{x},\mathds{1}\rangle^2+\frac{||x||_2^2 d}{2(n-1)}\\
			&\leq \frac{1}{2}\sum_{i=2}^{n}\lambda_2 \langle x,v_i\rangle^2+d\leq \frac{\lambda_2}{2}||x||_2^2+d= \frac{\lambda_2}{2}|U|+d.
		\end{align}
	
	\end{proof}

    We point out that the positive discrepancy may be much smaller than $\lambda_2 n$. Indeed, if $G$ is the disjoint union of the complete graph on $d+1$ vertices, and a random $d$-regular graph on $n\geq d^2$ vertices, then $\lambda_2=d$, but $\disc^{+}(G)=\Theta(d^{1/2}n)$. In general, $\disc^{+}(G)$ is a global quantity of $G$, while $\lambda_2$ may be influenced by small subgraphs as well.
    
    Finally, let us prove our upper bound on the positive discrepancy, that is, Proposition \ref{prop:upper_bound}.

    \begin{proof}[Proof of Proposition \ref{prop:upper_bound}]
    For a prime power $q$, let $Q=Q(q)$ be the strongly-regular graph defined in the introduction. That is, $Q$ has $N:=\frac{1}{2}q^2(q^2-1)$ vertices, it is $D=(q-1)(q^2+1)$-regular, and has second largest eigenvalue $\lambda_2=q-1$. Assuming $q$ is sufficiently large, we have $\lambda_2\leq 2N/D$ and $1.5N^{3/4}\leq D\leq 2N^{3/4}$. Write $p=D/(N-1)$ for the density of $Q$.
    
    Assume that $\varepsilon\leq 0.1$ and that $n$ is sufficiently large with respect to $\varepsilon$. We can choose a prime $q$ such that $N/D\in [(1-\varepsilon)\frac{n}{d},(1+\varepsilon)\frac{n}{d}]$  (see \cite{primes}). Then, using that $d\in [n^{2/3},n^{3/4}]$, we can write 
    $$\frac{N^{1/4}}{2}\leq \frac{N}{D}\leq 1.1\frac{n}{d}\leq 1.1n^{1/3},$$ which implies that $N\leq 25n^{4/3}$ and $D\leq 25n$. Also, $$0.7N^{1/4}\geq \frac{N}{D}\geq 0.9\frac{n}{d}\geq0.9n^{1/4},$$ which implies that $n\leq N$.
    
    Let $G$ be an arbitrary induced subgraph of $Q$ on $n$ vertices with density at least $p$. As the expected value of density of $n$ vertex induced subgraphs is $p$, there exists such a subgraph. The average degree of $G$ is at least $p(n-1)\geq (1-\varepsilon)d$. Also, $e(G)\leq p\binom{n}{2}+O(\lambda_2 n)\leq (1+\varepsilon/2)p\binom{n}{2}$, so the average degree is at most $(1+\varepsilon)d$. Finally, by Lemma \ref{lemma:lambda_2_disc}, for every $U\subset V(G)$, one has 
    $$\disc_G(U)\leq \disc_Q(U)\leq \frac{\lambda_2}{2} |U|+D\leq\frac{N}{D}\cdot n+25n\leq \frac{3n^2}{d}.$$
    Therefore, $\disc^{+}(G)\leq \frac{3n^2}{d}.$
    \end{proof}
	
	\subsection{Surplus}
	
	Let $G$ be a graph with $n$ vertices and $m$ edges. The \emph{surplus} of $G$, denoted by $\surplus(G)$, is the maximum $s$ such that $V(G)$ has a partition into two parts with $m/2+s$ edges between the parts. An \emph{equipartition} of $G$ is a partition of the vertices into parts of size $\lfloor n/2\rfloor$ and $\lceil n/2\rceil$. The \emph{negative surplus}, denoted by $\dfc(G)$, is the maximum $s$ such that $G$ has an equipartition with $m/2-s$ edges between the parts.
	
	In what follows, we prove that $\disc^{+}(G)=\Theta(\dfc(G))$ holds for regular graphs. Note that this does not necessarily remain true without the assumption of regularity: if $G$ is the union of a clique of size $n/2$  and an independent set of size $n/2$, with all edges added between the two sets, then $\disc^{+}(G)=\Omega(n^2)$ and $\dfc(G)=O(n)$.
	
	\begin{lemma}\label{lemma:surplus}
		Let $G$ be a $d$-regular graph on $n$ vertices, where $n\geq 100$. Then 
		\begin{itemize}
			\item $\frac{1}{2}\cdot \dfc(G)\leq \disc^{+}(G)\leq 3\cdot \dfc(G),$
			\item $\frac{1}{2}\cdot \surplus(G)\leq \disc^{-}(G)\leq 3\cdot \surplus(G).$
		\end{itemize}
	\end{lemma}
	
	\begin{proof}
		We only prove the first bulletpoint, the second follows in a similar manner. Also, for simplicity, we assume that $n$ is even. 
		
		Let $p=e(G)/\binom{n}{2}=\frac{d}{n-1}$. Let $U\subset V(G)$  such that $|U|=n/2$, and  achieving the equality  $e(U,U^c)=\frac{dn}{4}-\dfc(G)$. Then  
		$$\disc^{+}(G)\geq \disc(U)=-\frac{1}{2}\disc(U,U^c)=\dfc(G)/2,$$ where the first equality follows by Claim \ref{claim:complement}.
		
		Now let us turn to the upper bound, and let $U\subset V(G)$ such that $\disc^{+}(G)=\disc(U)$. As $\disc(U)=\disc(U^c)$, we may assume without loss of generality that $|U|\leq n/2\leq |U^c|$. Write $u:=|U|$, let $Z$ be a random $z:=n/2-u$ element subset of $U^c$, and set $Y=U\cup Z$.  We have
		$$\disc(Y)=\disc(U)+\disc(U,Z)+\disc(Z).$$
		Here, 
		\begin{align*}
			\mathbb{E}(\disc(U,Z))&=e(U,U^c)\cdot \frac{z}{n-u}-p\cdot uz
			=\frac{z}{n-u}(e(U,U^c)-p\cdot u(n-u))\\
			&=\frac{z}{n-u}\disc(U,U^c)=-2\frac{z}{n-u}\disc(U).
		\end{align*}
		and similarly 
		$$\mathbb{E}(\disc(Z))=e(U^c)\cdot\frac{z(z-1)}{(n-u)(n-u-1)}-p\binom{z(z-1)}{2}=\frac{z(z-1)}{(n-u)(n-u-1)}\disc(U^c).$$
		Using that $\disc(U^c)=\disc(U)$, we have
		\begin{align*}
			\mathbb{E}(\disc(Y))&=\disc(U)\left[1-\frac{2z}{n-u}+\frac{z(z-1)}{(n-u)(n-u-1)}\right]\\
			&=\disc(U)\left[\left(1-\frac{z}{n-u}\right)^2-\frac{nz}{2(n-u)^2(n-u-1)}\right]\geq \frac{1}{5}\disc(U).
		\end{align*}
		where the last inequality holds by noting that $\frac{z}{n-u}=\frac{n/2-u}{n-u}\leq \frac{1}{2}$,  and $\frac{nz}{2(n-u)^2(n-u-1)}<0.1$ by our assumption $n\geq 100$. This means that there is a choice for $Z$ such that $\disc(Y)\geq \frac{1}{5}\disc(U)$, fix such a choice. 
		
		As $|Y|=n/2$, we have
		$$\frac{nd}{4}-\dfc(G)\leq e(Y,Y^c)=\disc(Y,Y^c)+\frac{dn}{4}=-2\disc(Y)+\frac{dn}{4}$$
		from which we deduce that $\dfc(G)\geq 2\disc(Y)\geq \frac{1}{3}\disc(U).$
		
	\end{proof}

 \section{Positive discrepancy in the sparse regime}\label{sect:sparse}

 In this section, we prove the first two inequalities of Theorem \ref{thm:main}. The proof is partially inspired by the semidefinite programming method used in recent results on the MaxCut \cite{BJS,Max-Cut-SDP,GJS}. However, this is somewhat different from the semidefinite approach we introduce in Section \ref{sect:semidefinite} to tackle the third inequality.

 The idea of the proof is to assign $n$-dimensional vectors to the vertices of $G$ which achieve a large value for the so called semidefinite relaxation of the positive discrepancy. Then, 
 use the celebrated Goemans-Williamson \cite{GW95} rounding technique. That is, we take a random half-space through the origin, and consider the set $U$ of vertices of $G$, whose assigned vector lies in this half-space. We show that with a suitable assignment of vectors, the expected positive discrepancy of $U$ is large. First, we consider the case when $G$ is not too far from being regular, and then use Lemma \ref{lemma:maxdeg} to conclude the general case.

 In what follows, let $G$ be a graph on $n$ vertices of average degree $d$, and let $\Delta$ be the maximum degree of $G$. We will use the following simple technical result.

 \begin{claim}\label{claim:approximation}
    If $\beta\in [0,\frac{\pi}{2}]$ and $t=\cos(\frac{\pi}{2}-\beta)\leq 0.5$, then $0\leq t\leq \beta\leq t+t^2$.
   \end{claim}

   \begin{proof}
   Note that $\beta=\arcsin(t)$ and $t\geq 0$, so we clearly have $\beta\geq t$. On the other hand, setting $f(t)=t+t^2-\arcsin(t)$, we have $f'(t)=1+2t-\frac{1}{\sqrt{1-t^2}}$, which is nonnegative for $t\in [0,0.5]$. Thus, $f(t)\geq f(0)=0$.
   \end{proof}

  \begin{lemma}\label{lemma:main_sparse}
      If $\Delta\leq 2d$, then
		$$\disc^{+}(G)>\begin{cases} c_3d^{1/2}n &\mbox{if }d \leq n^{2/3},\\
			c_3n^2/d & \mbox{if }n^{2/3}<d\leq n/2,\end{cases}$$
     where $c_3>0$ is some absolute constant.
    \end{lemma}

    \begin{proof}

    Let $\delta=10^{-4}$, let $z= \frac{\delta}{\sqrt{d}}$ if $d\leq n^{2/3}$, and $z=\frac{\delta n}{d^2}$ if $d> n^{2/3}$. In both cases, $z\leq \frac{\delta}{\sqrt{d}}$. For $v\in V(G)$, define the vector $x_v\in\mathbb{R}^{V(G)}$ such that for $u\in V(G)$, 
   $$x_v(u)=\begin{cases}1 &\mbox{ if } u=v,\\
                        z &\mbox{ if } uv\in E(G),\\
                        0 &\mbox{ otherwise.}
   \end{cases}$$
   Note that $||x_v||_2\leq (1+\Delta z^2)^{1/2}\leq \sqrt{2}$. Also, let $y_v=x_v/||x_v||_2$, and let $\alpha_{u,v}$ be the angle between $y_u$ and $y_v$. Then, $\langle y_u,y_v\rangle=\cos(\alpha_{u,v})$. Let $\textbf{w}$ be a random unit vector in $\mathbb{R}^{V(G)}$, chosen from the uniform distribution, and define the set
   $$U=\{v\in V(G):\langle y_v,\textbf{w}\rangle\geq 0\}.$$
   Let us calculate the expectation of $\disc^{+}(U)$. For $v\in V(G)$, let $X_v$ be the indicator random variable of the event $v\in U$. Clearly, $\mathbb{E}(X_v)=\frac{1}{2}$. On the other hand, for $u,v\in V(G)$,
   $$\mathbb{E}(X_uX_v)=\frac{\pi-\alpha_{u,v}}{2\pi}.$$
   Indeed, let $H$ be the plane spanned by $y_u$ and $y_v$, and let $\textbf{w}=\textbf{w}_0+\textbf{w}_1$, where $\textbf{w}_0\in H$ and $\langle \textbf{w}_0,\textbf{w}_1\rangle=0$. Then $X_u=X_v=1$ if and only if $\textbf{w}_0$ is contained in the intersection of the half-planes of $H$ with normal vectors $y_u$ and $y_v$. This intersection forms an angle of size $\pi-\alpha_{u,v}$. Therefore the probability that $\textbf{w}_0$ is contained in this intersection is $\frac{\pi-\alpha_{u,v}}{2\pi}$ for every fixed $\textbf{w}_1$.
   
   Writing $\beta_{u,v}=\frac{\pi}{2}-\alpha_{u,v}$, we thus have $\mathbb{E}(X_uX_v)=\frac{1}{4}+\frac{\beta_{u,v}}{2\pi}$. Therefore,
   $$2\mathbb{E}(\disc(U))=\sum_{u\sim v}\mathbb{E}(X_{u}X_{v})-\frac{d}{n-1}\sum_{u\neq v} \mathbb{E}(X_uX_v)=\frac{1}{2\pi}\left(\sum_{u\sim v}\beta_{u,v}-\frac{d}{n-1}\sum_{u\neq v}\beta_{u,v}\right).$$
   Here and later,  $\sum_{u\sim v}$,  and $\sum_{u\neq v}$ denote the sum over all ordered pairs $(u,v)$ of vertices, with the appropriate additional condition; $u\sim v$ denotes that $\{u,v\}\in E(G)$. 
   
   Observe that for any $u\neq v$, we have $\langle y_u,y_v\rangle \leq \delta$, so by Claim \ref{claim:approximation}, $\langle y_u,y_v\rangle\leq \beta_{u,v}\leq (1+\delta)\langle y_u,y_v\rangle$. Hence,
   $$\frac{1}{2\pi}\left(\sum_{u\sim v}\beta_{u,v}-\frac{d}{n-1}\sum_{u\neq v}\beta_{u,v}\right)\geq \frac{1}{2\pi}\left(\sum_{u\sim v}\langle y_u,y_v\rangle-\frac{d}{n-1}\sum_{u\neq v}\langle y_u,y_v\rangle\right)-\frac{\delta d}{n-1}\sum_{u\neq v}\langle y_u,y_v\rangle.$$
   Note that if $u\sim v$, then $\langle y_u,y_v\rangle\geq \frac{2z}{||x_u||_2 ||x_v||_2}\geq z$. On the other hand, for every $u\not\sim v$, $\langle y_u,y_v\rangle\leq d(u,v)z^2$, where $d(u,v)$ is the number of common neighbours of $u$ and $v$. But by a simple double counting argument, $\sum_{u\neq v}d(u,v)\leq n\Delta^2$, so the right hand side of the previously highlighted  inequality is 
   $$\left(\frac{1}{2\pi}-\frac{d}{2\pi (n-1)}-\frac{\delta d}{n-1}\right)\cdot\sum_{u\sim v}\langle y_u,y_v\rangle-\left(\frac{d}{n-1}+\frac{\delta d}{n-1}\right)\cdot \sum_{u\not\sim v}\langle y_u,y_v\rangle\geq \frac{zdn}{30}-\frac{2d}{n} (n\Delta^2) z^2,$$
   by using the fact that $d\leq n/2$. Recalling that $\delta=10^{-4}$  and that $\Delta\leq 2d$, we thus get
   $$2\mathbb{E}(\disc(U))\geq \frac{zdn}{30}-8d^3 z^2\geq \begin{cases}
       \frac{\delta}{40}\cdot d^{1/2}n &\mbox{ if } d\leq n^{2/3}\\
       \frac{\delta}{40}\cdot \frac{n^{2}}{d} &\mbox{ if } d\geq n^{2/3}.
   \end{cases}.$$
   \end{proof}

   Now we are ready to prove the first two inequalities of Theorem \ref{thm:main}, which we restate as follows.

   \begin{theorem}\label{thm:main_sparse}
      There exists some absolute constant $c>0$ such that
		$$\disc^{+}(G)>\begin{cases} cd^{1/2}n &\mbox{if }d \leq n^{2/3},\\
			cn^2/d & \mbox{if }n^{2/3}<d\leq n/2.\end{cases}$$
     
    \end{theorem}

    \begin{proof}
        Let $c_1,c_2$ be the constants given by Lemma \ref{lemma:maxdeg} with respect to $\delta=0.1$, and let $c_3$ be the constant given by Lemma \ref{lemma:main_sparse}. Let $X$ be the set of vertices of degree more than $(1+\delta)d$, and let $f=e(X)+e(X,X^c)$. If $f\geq \frac{\delta}{10} dn$, we have 
  $$\disc^{+}(G)\geq c_1f-c_2n=\Omega(\delta dn),$$ and we are done. Hence, we may assume $f<\frac{\delta}{10} dn$. Let $G_0$ be the graph we get after removing all edges containing some vertex of $X$. Then $e(G_0)=e(G)-f$, the average degree of $G$ is $ d_0:=(dn-2f)/n\geq (1-\delta)d$, and its density is $p_0=p-f/\binom{n}{2}$. Observe that the maximum degree of $G_0$ is at most $(1+\delta)d<2d_0$. Therefore, by Lemma \ref{lemma:main_sparse}, we have
		$$\disc^{+}(G_0)>\begin{cases} c_3d_0^{1/2}n &\mbox{if }d_0 \leq n^{2/3},\\
			c_3n^2/d_0 & \mbox{if } n^{2/3}\leq d_0\leq n/2. \end{cases}$$
  Let $U\subset V(G)$ be a set attaining $\disc_{G_0}(U)=\disc^{+}(G_0)$. Then
  \begin{align*}|\disc_{G_0}(U)-\disc_{G}(U)|&=\left|e_{G_0}(U)-e_{G}(U)+\frac{d-d_0}{n-1}\binom{|U|}{2}\right|
  \leq f+\frac{2f}{n(n-1)}\cdot \frac{n(n-1)}{2}=2f.
  \end{align*}
		Thus, if $\disc^{+}(G_0)\geq 4f$, we have $\disc^{+}(G)\geq \frac{1}{2}\disc^{+}(G_0)$. On the other hand, if $4f>\disc^{+}(G_0)$, then $$\disc^{+}(G)\geq c_1f-c_2n\geq \frac{c_1}{4}\disc^{+}(G_0)-c_2n.$$
	Hence, combining this with our lower bounds on $\disc^{+}(G_0)$, and recalling that $d\geq d_0\geq d/2$, we get
 $$\disc^{+}(G)\geq \begin{cases} cd^{1/2}n &\mbox{if }d\leq n^{2/3},\\
			c n^2/d & \mbox{if }n^{2/3}\leq d\leq n/2,\end{cases}$$
   with some constant $c>0$.  \end{proof}
	
   \section{Semidefinite programming}\label{sect:semidefinite}
	The goal of this section is to introduce a semidefinite relaxation of $\disc^{+}(G)$, which provides a good approximation based on the so called graph Grothendieck's inequality \cite{AMMN,CW04}.

	Let $G$ be a graph with $n$ vertices, $m$ edges, density $p$, and recall that for $U\subset V(G)$, we have $\disc(U)=e(U)-p\binom{|U|}{2}$. Let us extend the definition of $\disc(.)$ for vectors. Without loss of generality, let the vertex set of $G$ be $\{1,\dots,n\}$, and for $\mathbf{x}\in\mathbb{R}^{n}$, let 
	$$\disc(\mathbf{x})=\frac{1}{2}\sum_{i\sim j}\mathbf{x}(i)\mathbf{x}(j)-\frac{p}{2}\sum_{i\neq j} \mathbf{x}(i)\mathbf{x}(j).$$

	Clearly, if $\mathbf{x}$ is the characteristic vector of $U$, we have $\disc(U)=\disc(\mathbf{x})$, and so $\disc^{+}(G)=\max_{\mathbf{x}\in \{0,1\}^n}\disc(\mathbf{x})$. Let us slightly modify the definition of $\disc^{+}$, and define
	$$\disc^{+}_1(G)=\max_{\mathbf{x}\in \{-1,1\}}\disc(\mathbf{x}).$$
	
	\begin{claim}
		$$\disc^{+}_1(G)\leq 4\disc^{+}(G).$$
	\end{claim}
	\begin{proof}
		Let $\mathbf{x}\in \{-1,1\}^{n}$ be a vector attaining the maximum of $\disc(.)$, and let $U=\{i:\mathbf{x}(i)=1\}$. Then 
		$$\disc^{+}_1(G)=\disc(\mathbf{x})=\disc(U)-\disc(U,U^c)+\disc(U^c).$$
		On the other hand, we also have
		$$0=\disc(U)+\disc(U,U^c)+\disc(U^c).$$
		Therefore, $\disc^{+}_1(G)=2\disc(U)+2\disc(U^c)$, which means that either $\disc(U)$ or $\disc(U^c)$ is at least $\disc^{+}_1(G)/4$.
	\end{proof}
	
	Observe that $\disc(\mathbf{x})$ is a multilinear function, so its maximum on $[-1,1]^n$ is attained by one of the extremal points. Therefore,
	$$\disc^{+}_1(G)=\max_{\mathbf{x}\in [-1,1]^{n}}\disc(\mathbf{x}).$$
	This also shows that $\disc^{+}(G)\leq \disc^{+}_1(G)$.
	
	Finally, we introduce the semidefinite relaxation of $\disc_1^{+}(G)$ as follows (and get rid of the factor $1/2$). Let $\pdisc^*(G)$ be the maximum of 
	$$\sum_{i\sim j}\langle v_i,v_j\rangle-p\sum_{i\neq j} \langle v_i,v_j\rangle$$
	over all $v_1,\dots,v_n\in\mathbb{R}^{n}$ such that $||v_i||_2\leq 1$. We trivially have $\pdisc^*(G)\geq 2\disc_1^{+}(G).$ On the other hand, $\pdisc^*(G)$ cannot be much smaller than $\disc_1^{+}(G)$ either. In order to show this, we use the following lemma. This lemma is a symmetric analogue of Grothendieck's inequality \cite{Groth}, a fundamental result in functional analysis.

	\begin{lemma}[\cite{AMMN,CW04}]\label{lemma:groth}
		There exists a universal constant $C>0$ such that the following holds. Let $M\in\Sym(n)$. Let 
		$$\beta=\sup_{\mathbf{x}\in [-1,1]^{n}}\sum M_{i,j}\mathbf{x}(i)\mathbf{x}(j),$$
		and let
		$$\beta^{*}=\sup_{v_1,\dots,v_n\in \mathbb{R}^{n},||v_i||_2\leq 1}\sum M_{i,j}\langle v_i,v_j\rangle.$$
		Then $\beta\leq \beta^{*}\leq C\beta \log n$.
	\end{lemma}
	
	If $A$ is the adjacency matrix of $G$, we can apply Lemma \ref{lemma:groth} with $M:=A-\frac{p}{2}(J-I)$ to conclude that $\pdisc^*(G)\leq C\disc_1^{+}(G)\log n$.
	
	In the rest of the paper, we rather directly work with the Gram matrix given by the $n$ vectors. That is, we will consider the following equivalent definition of $\pdisc^*(G):$
	\begin{align*}
		\max_{X\in \Sym(n)}\ \ \ &\sum_{i\sim j}X_{i,j}-p\sum_{i\neq j}X_{i,j}\\
		\mbox{\textbf{subject to }}\ \ &X_{i,i}\leq 1 \mbox{ for all }i,\\
		&X\succeq 0.
	\end{align*} 
 
    It is actually slightly more convenient to work with the following modified version of $\pdisc^*(G)$. For a matrix $X\in\Sym(n)$, write 
	$$\disc(X)=\sum_{i\sim j}X_{i,j}-\frac{d}{n}\sum X_{i,j}=\left\langle X,A-\frac{d}{n}J\right\rangle,$$
	and let $\pdisc(G)$ be the solution of
	\begin{align*}
		\max_{X\in \Sym(n)}\ \ \ &\disc(X)\\
		\mbox{\textbf{subject to }}\ \ &X_{i,i}\leq 1 \mbox{ for all }i,\\
		&X\succeq 0.
	\end{align*}
	Note that 
	$$\disc^*(X)-\disc(X)=\frac{d}{n}\sum_i X_{ii}+\frac{d}{n(n-1)}\sum_{i \not = j}X_{ij}=\left\langle X,\frac{d}{n-1}I-\frac{d}{n(n-1)}J\right\rangle,$$
	so $\pdisc(G)\leq \pdisc^*(G)\leq \pdisc(G)+2d$. The first inequality follows as $I-\frac{1}{n}J$ is positive semidefinite and the second from the fact that since $X\succeq 0, X_{ii} \leq 1$, then
 all $X_{i,j} \leq 1$. To summarize this section so far, we arrive to the following relationship between $\pdisc(G)$ and $\disc^{+}(G)$.
	
	\begin{lemma}\label{lemma:disc_to_disc}
		There exists a constant $c>0$ such that $$2d+\pdisc(G)\geq \disc^{+}(G)\geq \frac{c\pdisc(G)}{\log n}.$$
	\end{lemma}
 
	\section{Positive discrepancy and eigenvalues}\label{sect:easy}
	
	In this section, we present a number of lower bounds on the positive discrepancy based on the eigenvalues of the adjacency matrix.   For the rest of this section and the next section, let $G$ be a graph on vertex set $\{1,\dots,n\}$ with average degree $d$, and maximum degree $\Delta$. Let $A$ be the adjacency matrix of $G$, let $\lambda_1\geq \dots\geq \lambda_n$ be the eigenvalues of $A$, and let $v_1,\dots,v_n$ be an orthonormal basis of corresponding eigenvectors. We will write $\mathbf{e}=\frac{1}{\sqrt{n}}\mathds{1}$. 
 
 The following inequalities between $d,\Delta,\lambda_1$ are well known, see e.g. Theorem 1.1 in the survey of Cvetković and Rowlinson~\cite{eigenvalue_survey}.
	
	\begin{claim}\label{claim:first_eigen}
		 $d\leq \lambda_1\leq \Delta$.
	\end{claim}
 


    Let us start with a useful identity.
	
	\begin{claim}\label{claim:identity1}
		For $i=1,\dots,n$, let $\alpha_i,\beta_i$ be real numbers, and let $X=\sum_{i=1}^{n}\alpha_i  v_{i}v_i^{T}$ and  $Y=\sum_{i=1}^{n}\beta_i  v_{i}v_i^{T}.$ Then
		$$\mathds{1}^{T}(X\circ Y)\mathds{1}=\langle X,Y\rangle=\sum_{i=1}^{n}\alpha_i \beta_i.$$
	\end{claim}
	
	\begin{proof}
		We have
		$$\langle X,Y\rangle=\sum_{1\leq i,j\leq n}\alpha_i\beta_j\langle v_{i},v_{j}\rangle^2=\sum_{i=1}^{n}\alpha_i\beta_i.$$
	\end{proof}
	
	\begin{corollary}\label{cor:identity2}
		Let $\alpha_1,\dots,\alpha_n\geq 0$ and let $X=\sum_{i=1}^{n}\alpha_i  v_{i}v_i^{T}$. Then
		$$\disc(X)\geq \sum_{i=1}^{n}\alpha_i\lambda_i-d\cdot \max_{i\in [n]}\alpha_i.$$
	\end{corollary}
	
	\begin{proof}
		Writing $J=n\cdot \mathbf{e}\mathbf{e}^T$, we have
		\begin{align*}
		 \disc(X)&=\langle X,A-d\mathbf{e}\mathbf{e}^T\rangle=\sum_{i=1}^{n}\alpha_i\lambda_i-d\sum_{i=1}^{n}\alpha_i\langle v_i,\mathbf{e}\rangle^2\\
		&\geq \sum_{i=1}^{n}\alpha_i\lambda_i-d\left(\max_{i\in [n]}\alpha_i\right)\cdot\sum_{i=1}^{n}\langle v_i,\mathbf{e}\rangle^2=\sum_{i=1}^{n}\alpha_i\lambda_i-d\cdot \max_{i\in [n]}\alpha_i.
	\end{align*}
	 Here, the second equality holds by Claim \ref{claim:identity1}, while the last equality holds as $v_1,\dots,v_n$ is an orthonormal basis.
	\end{proof}
	
	In the next three lemmas, we present lower bounds on $\pdisc(G)$ with the help of the positive eigenvalues of $A$. Let $K$ denote the largest index $i$ such that $\lambda_{i}>0$.

	\begin{lemma}\label{lemma:sum_eigenvalues}
		$$\pdisc(G)\geq \sum_{i=2}^{K}\lambda_i.$$
	\end{lemma}
	
	\begin{proof}
		Set
		$X=\sum_{i=1}^{K}v_iv_i^{T},$
		then $0\preceq X\preceq I$. The upper bound implies that $X_{i,i}\leq 1$ for $i\in [n]$. Therefore, by Corollary \ref{cor:identity2},
		$$\pdisc(G)\geq \disc(X)\geq \sum_{i=1}^K\lambda_i-d\geq \sum_{i=2}^{K}\lambda_i.$$
	\end{proof}

        The sum $\sum_{i=1}^{n}|\lambda_i|$ is known as the \emph{energy} of the graph $G$, and it is denoted by $\mathcal{E}(G)$. As $\sum_{i=1}^{n}\lambda_i=0$, we have that $\sum_{i=2}^{K}\lambda_i=\frac{1}{2}\mathcal{E}(G)-\lambda_1$. Hence, Lemma \ref{lemma:sum_eigenvalues} immediately implies the following relation between the energy and the positive discrepancy, which might be of independent interest.

        \begin{corollary}
            $$\pdisc(G)\geq \frac{1}{2}\cdot \mathcal{E}(G)-\lambda_1.$$
        \end{corollary}

        Next, we present a lower bound on $\pdisc(G)$ using the cubic sum of positive eigenvalues.
	
	\begin{lemma}\label{lemma:cube_sum_eigenvalues}
		$$\pdisc(G)\geq \frac{1}{\Delta}\sum_{i=2}^{K}\lambda_i^{3}.$$
	\end{lemma}
	
 \begin{proof}
		Set
		$$Y=\sum_{i=1}^{K}\lambda_i^2 v_iv_i^{T} \preceq \sum_{i=1}^{n}\lambda_i^2 v_iv_i^{T}=A^2.$$
		Clearly, $Y\succeq 0$ and since $Y\preceq A^2$, every diagonal entry of $A^2$ is at most $\Delta$. Therefore, $Y_{i,i}\leq \Delta$ for every $i\in [n]$. Hence, we get
		$$\pdisc(G)\geq \disc\left(\frac{1}{\Delta}\cdot Y\right)\geq \frac{1}{\Delta}\left(\sum_{i=1}^{K}\lambda_i^{3}-d\lambda_1^2\right)\geq \frac{1}{\Delta}\sum_{i=2}^{K}\lambda_i^{3},$$
		where the second inequality follows from Corollary \ref{cor:identity2}.
	\end{proof}

        Combining Lemmas \ref{lemma:sum_eigenvalues} and \ref{lemma:cube_sum_eigenvalues}, we can deduce a lower bound using the squared sum of positive eigenvalues.
 
	\begin{lemma}\label{lemma:square_sum_eigenvalues}
		$$\pdisc(G)\geq \frac{1}{\sqrt{\Delta}}\sum_{i=2}^{K}\lambda_i^{2}.$$
	\end{lemma}
	
	\begin{proof}
		Let $a_1,\dots,a_k$ be nonnegative real numbers. Then 
		$$(a_1+\dots+a_k)(a_1^3+\dots+a_k^3)\geq (a_1^2+\dots+a_k^2)^2.$$
		Hence, by multiplying the  bounds in Lemma \ref{lemma:sum_eigenvalues} and Lemma \ref{lemma:cube_sum_eigenvalues}, we get
		$$\pdisc(G)^2\geq \frac{1}{\Delta}\left(\sum_{i=2}^{K}\lambda_i\right)\left(\sum_{i=2}^{K}\lambda_i^3\right)\geq \frac{1}{\Delta}\left(\sum_{i=2}^{K}\lambda_i^2\right)^2.$$
	\end{proof}

   The following quantity plays a crucial role in all of our lower bound arguments. Let
   $$\Lambda=\Lambda(G):=-\sum_{i=2}^{n}\lambda_i^3.$$
   Observe that $0\leq \mbox{tr}(A^3)=\sum_{i=1}^{n}\lambda_i^3=\lambda_1^3-\Lambda$, and $\lambda_1\leq \Delta$, so we immediately have the following upper bound on $\Lambda$.
   
   \begin{claim}\label{claim:max_Lambda}
   $\Lambda\leq \Delta^3$
   \end{claim}

   In what follows, we prove that $\Lambda$ being small implies large positive discrepancy.

   \begin{lemma}\label{lemma:Lambda}
   There exists $c>0$ such that the following holds. Let $d\leq n/2$ and $\Delta\leq 1.1d$, then
		$$\pdisc(G)>\begin{cases} cd^{1/2}n &\mbox{if }\Lambda \leq \frac{1}{16} d^{3/2}n,\\
			\frac{cd^2 n^2}{\Lambda}-\Delta & \mbox{if }\Lambda >\frac{1}{16} d^{3/2}n.	\end{cases}$$
   \end{lemma}

   \begin{proof}
    First, let us assume that $\Lambda_2:=\sum_{i=K+1}^{n}\lambda_i^2\leq dn/4$. Then, using  $dn=\mbox{tr}(A^2)=\sum_{i=1}^{n}\lambda_i^2$, we deduce that 
   $$\sum_{i=2}^{K}\lambda_i^2\geq \frac{3dn}{4}-\lambda_1^2\geq \frac{3dn}{4}-\Delta^2\geq \frac{dn}{4}.$$
  Hence, $\pdisc(G)\geq d^{1/2}n/5$ by Lemma \ref{lemma:square_sum_eigenvalues}, so we get the desired inequality independently of the value of $\Lambda$.

   Therefore, we may assume that $\Lambda_2>dn/4$. Let $\Lambda_3=\sum_{i=K+1}^{n}|\lambda_i|^3$. Let us first consider the case when $\Lambda \leq \frac{1}{2}\Lambda_3$. Then $\sum_{i=2}^{K} \lambda_i^3  = \Lambda_3  - \Lambda \geq \frac{1}{2}\Lambda_3$, which together with Lemma \ref{lemma:cube_sum_eigenvalues} implies $\pdisc(G) \geq \frac{\Lambda_3}{2\Delta}$. By the inequality between the quadratic and cubic mean, we have
   $$\left(\frac{\Lambda_2}{n}\right)^{1/2} \leq \left(\frac{\sum_{i=K+1}^{n}\lambda_i^2}{n}\right)^{1/2}\leq \left(\frac{\sum_{i=K+1}^{n} |\lambda_i|^3}{n}\right)^{1/3}=\left(\frac{\Lambda_3}{n}\right)^{1/3},$$
   which in turn implies that 
   $$\pdisc(G) \geq \frac{\Lambda_3}{\Delta} \geq \frac{\Lambda_2^{3/2}}{\Delta \sqrt{n}} \geq cd^{1/2}n$$ for some constant $c$.

   Now suppose that $\Lambda > \frac{1}{2}\Lambda_3 \geq  \frac{1}{2\sqrt{n}} \Lambda_2^{3/2} \geq \frac{1}{16}d^{3/2}n$ and let $\Lambda_1=\sum_{i=K+1}^{n}|\lambda_i|$. Just as in the proof of Lemma \ref{lemma:square_sum_eigenvalues}, we can write $\Lambda_1\Lambda_3\geq \Lambda_2^2.$  But $0=\mbox{tr}(A)=\sum_{i=1}^{K}\lambda_i-\Lambda_1$, so replacing $\Lambda_1$ by $\sum_{i=1}^{K}\lambda_i$, we get the inequality
   $$\sum_{i=1}^{K}\lambda_i\geq \frac{\Lambda_2^2}{2\Lambda}\geq \frac{d^2n^2}{32\Lambda}.$$ 
   Using Lemma \ref{lemma:sum_eigenvalues} and $\lambda_1\leq \Delta$, we conclude
   $$\pdisc(G)\geq \sum_{i=2}^{K}\lambda_i\geq \frac{d^2n^2}{32\Lambda}-\Delta,$$
   finishing the proof.
   \end{proof}



	
	\section{Positive discrepancy in dense regime}\label{sect:hard}

	In this section, we prove the following theorem, which gives the best lower bound on the positive discrepancy when $d$ is not much smaller than $n$, addressing the third inequality in Theorem \ref{thm:main}. At the end of the section, we then combine this with Lemma \ref{lemma:maxdeg} to prove the third inequality in Theorem \ref{thm:main}.
	
	\begin{theorem}\label{thm:difficult}
		For every $\varepsilon>0$ there exist $c>0$ such that the following holds. Let $G$ be a graph on $n$ vertices of average degree $d$ and maximum degree $\Delta\leq d(1+\varepsilon/2)$ such that $d\leq n(1/2-\varepsilon)$. Then 
		$$\pdisc(G)\geq cd^{1/4}n.$$
	\end{theorem}
	
	This theorem implies the third inequality in Theorem \ref{thm:main} after a bit of work. The proof of Theorem \ref{thm:difficult} is fairly involved. To aid the reader, let us give a brief outline of our approach.
	
	\medskip
	
	\noindent
	\textbf{Proof outline.} We may assume that $\Lambda\geq d^{7/4}n$, otherwise we are done by Lemma \ref{lemma:Lambda}. We will study the matrix $Y=\frac{1}{\Delta}\sum_{i=2}^{n}\lambda_i^2 v_iv_i^T$, which satisfies $\langle A,Y\rangle=-\frac{\Lambda}{\Delta}$. The idea is that $-\langle A,Y\rangle$ being large suggests that the entries of $Y$ concentrated on edges of $G$ are mostly very negative. Thus, taking $X=Y\circ (I-\frac{1}{t}A)$ with a suitable $t>0$, we get a matrix whose entries corresponding to edges are mostly very positive. Unfortunately, $I-\frac{1}{t}A$ is not necessarily positive semidefinite, as the eigenvalues  of $A$ larger than $t$ contribute negative eigenvalues. In order to fix this, we add a term $E$ to correct these eigenvalues, and study $X=Y\circ(I-\frac{1}{t}A+E)$ instead. Then the Schur product theorem (Lemma \ref{lemma:hadamard}) ensures that $X$ is also positive semidefinite. All it remains is to guarantee that the contribution from this new term $E$ does not decrease the discrepancy by too much, and that the diagonal entries of $X$ are not too large.
	
	\medskip
	
	Let us execute this strategy formally and prove Theorem \ref{thm:difficult}. We may assume that $\pdisc(G)\leq d^{1/4}n$ and $\Lambda\geq d^{7/4}n$, otherwise we are done by Lemma \ref{lemma:Lambda}.  Furthermore, we assume that the maximum degree is $\Delta\leq (1+\varepsilon/2)d$. We continue using the notation introduced in the previous section.

 First, we show that $v_1$, the eigenvector corresponding to the largest eigenvalue $\lambda_1$, must be close to $\mathbf{e}$ if the positive discrepancy is small. Recall that by the Perron-Frobenius theorem, we can assume that every coordinate of $v_1$ is nonnegative. To this end, let $w=v_1-\langle v_1,\mathbf{e}\rangle\cdot  \mathbf{e}$.

 \begin{lemma}\label{cor:w_upper}
 $||w||_2 \leq (2n)^{1/2}/d^{7/8}.$
 \end{lemma}

 \begin{proof}
 We prove that $$\pdisc(G)\geq \frac{d^3}{\Delta}||w||_2^2,$$ then the lemma follows by plugging in $\Delta\leq 2d$ and $\pdisc(G)\leq d^{1/4}n$.
 
 Let $u=\frac{v_1}{||v_1||_{\infty}}$, and set 
 $$X=uu^{T}=\frac{1}{||v_1||_{\infty}^2}\cdot v_1v_1^{T}.$$ Then $X\succeq 0$ and  $X_{i,i}=u(i)^2\leq 1$ for every $i\in [n]$. Hence,
 $$\pdisc(G)\geq \disc(X)=\langle X,A\rangle-\frac{d}{n}\langle X,J\rangle=\frac{1}{||v_1||_{\infty}^2}(\lambda_1-d\langle v_1,\mathbf{e}\rangle^2)=\frac{1}{||v_1||_{\infty}^2}((\lambda_1-d)+d||w||_2^2).$$
 Here, the second equality follows by Claim \ref{claim:identity1}, and the last equality by $1=||v_1||^2_2=||w||^2_2+\langle v_1,\mathbf{e}\rangle^2$. Using that $\lambda_1\geq d$, we get $$\pdisc(G)\geq \frac{d}{||v_1||_{\infty}^2}||w||_2^2.$$ 
 It remains to bound $||v_1||_{\infty}$. Let $k\in [n]$ be an index such that $v_1(k)=||v_1||_{\infty}$, then considering the equality $Av_1=\lambda_1 v_1$, we get $$\lambda_1 v_1(k)=\sum_{i\sim k}v_1(i)\leq \sqrt{\Delta}\left(\sum_{i\sim k} v_1(i)^2\right)^{1/2}\leq \sqrt{\Delta},$$
 where the first inequality is due to the Cauchy-Schwartz inequality. Hence, $||v_1||_{\infty}\leq \frac{\sqrt{\Delta}}{d}$, finishing the proof.
 \end{proof}

	In what follows, the main target of our focus is the matrix $$Y:=\frac{1}{\Delta}\sum_{i=2}^{n}\lambda_i^2 v_iv_i^{T}.$$ Let us collect the important properties of $Y$. 
 
 \begin{claim}
     \begin{itemize}
         \item[1.] $Y$ is positive semidefinite,
         \item[2.] $Y_{i,i}<1$ for $i\in [n]$,
         \item[3.] $\langle  Y,A\rangle=-\frac{\Lambda}{\Delta}$.
     \end{itemize}
 \end{claim}
 \begin{proof}
 1. is trivial; 2. holds by noting that $Y\preceq \frac{1}{\Delta}A^2$; 3. holds as
  $$\langle Y,A\rangle=\frac{1}{\Delta}\sum_{i=2}^{n}\lambda_i^{3}=-\frac{\Lambda}{\Delta},$$
  where the first equality is due to Claim \ref{claim:identity1}.
 \end{proof}

 Next, we will prove that $||Y\circ A||_{F}$ is not too large, assuming $\pdisc(G)$ is small.
	
	\begin{claim}\label{claim:YoA}
		$$||Y\circ A||_{F}^2< \frac{2\Lambda}{\Delta}.$$
	\end{claim}
	
	\begin{proof}
		Let us consider the matrix $X:=Y\circ Y$. We have $X\succeq 0$ by Lemma \ref{lemma:hadamard}, and $X_{i,i}\leq 1$ for every $i\in [n]$, therefore $\pdisc(G)\geq \disc(X)$. Let us try to estimate 
		$$\disc(X)=\langle X,A\rangle-\frac{d}{n}\langle X,J\rangle=||Y\circ A||_{F}^2-\frac{d}{n}||Y||_{F}^2.$$
		We consider the two terms on the right-hand-side separately. For the second term, note that $Y=\frac{1}{\Delta}\sum_{i=2}^{n}\lambda_{i}^2v_iv_i^{T}$, so 
		\begin{align*}
			||Y||_{F}^2&=\frac{1}{\Delta^2}\sum_{i=2}^{n}\lambda_i^4
			\leq \frac{1}{\Delta}\sum_{i=2}^{K}\lambda_i^3+\frac{1}{\Delta}\sum_{i=K+1}^{n}-\lambda_i^3
			=\frac{2}{\Delta}\sum_{i=2}^{K}\lambda_i^3+\frac{1}{\Delta}\sum_{i=2}^{n}-\lambda_i^3\\
			&\leq 2\pdisc(G)+\frac{\Lambda}{\Delta}\leq \frac{2\Lambda}{\Delta}.
		\end{align*} 
		Here, the first inequality holds by $|\lambda_i|\leq \Delta$ for every $i\in [n]$, the second inequality holds by Lemma~\ref{lemma:cube_sum_eigenvalues}, and the last inequality holds by our assumption on $\pdisc(G)$ and $\Lambda$.

		Suppose that $||Y\circ A||_{F}^2\geq \frac{2\Lambda}{\Delta}$. Then we have
		$$\disc(X)\geq \frac{2\Lambda}{\Delta}-\frac{d}{n}||Y||_F^2\geq \frac{2\Lambda}{\Delta}-\frac{2d\Lambda }{n\Delta}\geq \frac{\Lambda}{\Delta},$$
		where the last inequality holds by the assumption $d<n/2$. But then $\pdisc(G)\geq \disc(X)\geq \frac{1}{2}d^{3/4}n$, contradiction.
	\end{proof}

	Let $t=d^{1/2}$, and define
	$$Z=Z(t):=I-\frac{1}{t}A.$$
	
	\begin{claim}\label{claim:RoZ}
		Let $R\in\Sym(n)$. Then 
		$$\disc(R\circ Z)=-\frac{d}{n}\sum_{i=1}^{n}R_{i,i}-\frac{1}{t}\left(1-\frac{d}{n}\right)\langle R,A\rangle.$$
	\end{claim}
	
	\begin{proof}
	 We have
		$$\disc(R\circ A)=\langle R\circ A,A\rangle-\frac{d}{n}\langle R\circ A,J\rangle=\left(1-\frac{d}{n}\right)\langle R,A\rangle.$$
		Hence, we can write 
		$$
			\disc(R\circ Z)=\disc(R\circ I)-\frac{1}{t}\disc(R\circ A)
			=-\frac{d}{n}\sum_{i=1}^{n}R_{i,i}-\frac{1}{t}\left(1-\frac{d}{n}\right)\langle R,A\rangle.$$
	\end{proof}
	
	From the previous claim, we get that $\disc(Z\circ Y)\approx -\frac{1}{t}\langle Y,A\rangle=\frac{\Lambda}{d^{1/2}\Delta}$, so if $\Lambda$ is large, the matrix $Z\circ Y$ has large discrepancy. Unfortunately, we did not guarantee that $Z$ is positive semidefinite, which would ensure that $Z\circ Y$ is also positive semidefinite by the Schur product theorem. In order to fix this, we add a correctional term. Set
	$$E:=\sum_{\lambda_i\geq t}\frac{\lambda_i}{t}\cdot v_iv_i^{T}.$$  
	Since $I=\sum_{i=1}^n v_iv_i^{T}$, we crucially have

	$$Z+E=\sum_{\lambda_i\geq t}v_iv_i^{T}+\sum_{\lambda_i<t}\left(1-\frac{\lambda_i}{t}\right)v_iv_i^{T}\succeq 0.$$
	Now we would like to ensure  that the addition of $E$ does not alter the discrepancy by much, that is, $\disc((Z+E)\circ Y)\approx \disc(Z\circ Y)$, and also that the diagonal entries of $E$ are not too large. Let us first address the latter issue.
	
	\begin{lemma}\label{lemma:max_diag}
		$E_{jj}\leq 1$ for every $j\in [n]$.
	\end{lemma}
	
	\begin{proof}
		Consider the matrix $U=(I+\frac{1}{\sqrt{d}}A)^2=
  \sum_{i=1}^n \big(1+\frac{\lambda_i}{\sqrt{d}}\big)^2 v_iv_i^{T}$. Clearly, $U_{jj}\leq 1+\Delta/d \leq 3$ for every $j$. Also since $(1+a)^2 \geq 4a$, we can write
		$$3\geq U_{jj}=\sum_{i=1}^{n}\left(1+\frac{\lambda_i}{\sqrt{d}}\right)^2v_i(j)^2\geq \sum_{\lambda_i>t}\left(1+\frac{\lambda_i}{\sqrt{d}}\right)^2v_i(j)^2\geq \sum_{\lambda_i>t}\frac{4\lambda_i}{\sqrt{d}}v_i(j)^2.$$
		Hence,
		$$E_{jj}=\sum_{ \lambda_i>t}\frac{\lambda_i}{t} v_i(j)^2 < \frac{\sqrt{d}}{t}.$$
	\end{proof}
	
 Now let us consider $\disc(E\circ Y)$. Let $E_0:=\frac{\lambda_1}{t}\cdot v_1v_1^{T}$ and $$E_1=\sum_{i\geq 2,\lambda_i\geq t}\frac{\lambda_i}{t}\cdot v_iv_i^{T},$$ and let us consider $\disc(E_0\circ Y)$ and $\disc(E_1\circ Y)$ separately.
 
    \begin{claim}\label{claim:E0}
    $\disc(E_0\circ Y)\geq \frac{\Delta}{tn}\langle Y,A\rangle-10n^{1/2}d^{5/8}.$
    \end{claim}

    \begin{proof}
    We have
    $$\disc((v_1v_1^{T})\circ Y)=\langle (v_1v_1^{T})\circ Y, A\rangle-\frac{d}{n}\langle (v_1v_1^{T})\circ Y,J\rangle.$$
    Here, the second term is equal to $\frac{d}{n}\langle v_1v_1^{T},Y\rangle=0$, using Claim \ref{claim:identity1}. Let us write $v_1=w+\alpha\mathbf{e}$, where $\alpha=\langle v_1,\mathbf{e}\rangle\leq 1$. Then
    $$\langle (v_1v_1^{T})\circ Y, A\rangle=\alpha^2 \langle (\mathbf{e}\mathbf{e}^{T})\circ Y,A\rangle+2\alpha\langle (\mathbf{e}w^{T})\circ Y,A\rangle+\langle (ww^{T})\circ Y, A\rangle.$$
    Here, the first term is equal to $\frac{\alpha^2}{n}\langle Y,A\rangle$. In the second term, we have
    $$|\langle (\mathbf{e}w^{T})\circ Y,A\rangle|=|\mathbf{e}^T(Y\circ A)w|\leq ||\mathbf{e}||_2\cdot||Y\circ A||_2\cdot ||w||_2\leq \Delta ||w||_2.$$
    Here, the inequality $||Y\circ A||_2\leq \Delta$ follows from the Gershgorin circle theorem: observe that the diagonal entries of $Y\circ A$ are 0, and each row contains at most $\Delta$ nonzero entries, each of absolute value at most 1. Using Lemma \ref{cor:w_upper}, the right hand side is at most $\Delta (2n/d^{7/4})^{1/2}\leq 2n^{1/2}d^{1/8}$. Finally, the term $\langle (ww^{T})\circ Y, A\rangle=w^T(Y\circ A)w$ can be similarly bounded by $2n^{1/2}d^{1/8}$. Therefore,
    $$\disc((v_1v_1^{T})\circ Y)\geq  \frac{\alpha^2}{n}\langle Y,A\rangle-6n^{1/2}d^{1/8}\geq \frac{1}{n}\langle Y,A\rangle-6n^{1/2}d^{1/8},$$
    where the second inequality uses that $\langle Y,A\rangle=-\Lambda/\Delta$ is negative. From this, it follows that
    $$\disc(E_0\circ Y)\geq \frac{\lambda_1}{nt}\cdot \langle Y,A\rangle-\frac{\lambda_1}{t}\cdot 6n^{1/2}d^{1/8}\geq \frac{\Delta}{tn}\langle Y,A\rangle-10n^{1/2}d^{5/8}.$$
    \end{proof}

    \begin{claim}\label{claim:E1}
    $|\disc(E_1\circ Y)|\leq 3\Lambda^{1/2}d^{-3/4}\pdisc(G)^{1/2}.$
    \end{claim}

    \begin{proof}
    We have
	$$\disc(E_1\circ Y)=\langle E_1\circ Y,A\rangle-\frac{d}{n}\langle E_1\circ Y,J\rangle=\langle E_1,Y\circ A\rangle-\frac{d}{n}\langle E_1,Y\rangle.$$
	Here, using Claim \ref{claim:identity1}, 
	$$\langle E_1,Y\rangle=\sum_{i\geq 2,\lambda_i\geq t} \frac{\lambda_i^2}{\Delta}\cdot \frac{\lambda_i}{t}=\frac{1}{\Delta t}\sum_{i\geq 2,\lambda_i\geq t}\lambda_i^3\leq \frac{\pdisc(G)}{d^{1/2}}<n,$$
	where the second inequality holds by Lemma \ref{lemma:cube_sum_eigenvalues}, and the last inequality holds by our assumption on $\pdisc(G)$. On the other hand, using the Cauchy-Schwartz inequality, we can write
	\begin{align*}
		|\langle E_1,Y\circ A\rangle|&\leq ||E_1||_{F}\cdot ||Y\circ A||_{F}
		=\left(\sum_{i\geq 2,\lambda_i\geq t}\left(\frac{\lambda_i}{t}\right)^2\right)^{1/2}\cdot ||Y\circ A||_{F}\\
		&\leq \frac{1}{t}\cdot \left(\sum_{i\geq 2,\lambda_i\geq t}\lambda_i^2\right)^{1/2}\cdot \left(\frac{2\Lambda}{d}\right)^{1/2}
		\leq 2\Lambda^{1/2} d^{-3/4}\pdisc(G)^{1/2}
	\end{align*}
	where the second inequality holds by Claim \ref{claim:YoA}, and the last inequality holds by Lemma \ref{lemma:square_sum_eigenvalues}. In conclusion, we get
	\begin{align*}
		\begin{split}
			|\disc(E_1\circ Y)|&\leq|\langle E_1,Y\circ A\rangle|+\frac{d}{n} |\langle E_1,Y\rangle|
			\leq 2\Lambda^{1/2}d^{-3/4}\pdisc(G)^{1/2}+d
			\leq 3\Lambda^{1/2}d^{-3/4}\pdisc(G)^{1/2}.
		\end{split}
	\end{align*}

 \end{proof}

 \begin{proof}[Proof of Theorem \ref{thm:difficult}]
 
	Set $X=(Z+E)\circ Y$, then $X\succeq 0$ by Lemma \ref{lemma:hadamard}, and $X_{i,i}=(Z_{i,i}+E_{i,i})Y_{i,i}\leq 2$ for every $i\in [n]$ by Lemma \ref{lemma:max_diag}. Also, we have
	\begin{align} \label{equ:2}
		\begin{split}
			\disc(X)&=\disc(Z\circ Y)+\disc(E_0\circ Y)+\disc(E_1\circ Y)\\
            &\geq \disc(Z\circ Y)+\disc(E_0\circ Y)-|\disc(E_1\circ Y)|\\
			&\geq -d-\frac{n-d}{tn}\cdot \langle Y,A\rangle+\left[\frac{\Delta}{tn}\langle Y,A\rangle-10n^{1/2}d^{5/8}\right]-|\disc(E_1\circ Y)|\\
            &\geq \frac{n-d-\Delta}{tn}\cdot (-\langle Y,A\rangle)-10nd^{1/8}-|\disc(E_1\circ Y)|\\
			&\geq \frac{\varepsilon \Lambda}{8d^{3/2}}-|\disc(E_1\circ Y)|,
		\end{split}
	\end{align}
	where the second inequality uses Claim \ref{claim:RoZ} and Claim \ref{claim:E0}, and last inequality uses $\langle Y,A\rangle=-\Lambda/\Delta$, $d\leq n(1/2-\varepsilon)$,  $\Delta\leq d(1+\varepsilon/2)$ and $\Lambda\geq d^{7/4}n$. Next, consider two cases.
	
	\begin{description}
		\item[Case 1.] $|\disc(E_1\circ Y)|\leq \frac{\varepsilon \Lambda}{16d^{3/2}}$.
		
		By (\ref{equ:2}), we have $$\disc(X)\geq \frac{\varepsilon \Lambda}{16d^{3/2}}.$$ 
		But as $X_{i,i}\leq 2$ for every $i\in [n]$,
		$$\pdisc(G)\geq \frac{1}{2}\disc(X)\geq \frac{\varepsilon \Lambda}{32d^{3/2}}.$$
		As $\Lambda\geq d^{7/4}n$, this implies $\pdisc(G)\geq \frac{\varepsilon d^{1/4}n}{32}$, finishing the proof of Theorem \ref{thm:difficult}.
		
		\item[Case 2.] $|\disc(E_1\circ Y)|> \frac{\varepsilon \Lambda}{16d^{3/2}}$.
		
		Comparing this lower bound on  $|\disc(E_1\circ Y)|$ with the upper bound of Claim \ref{claim:E1}, we get
		the inequality
		$$ 3\Lambda^{1/2}d^{-3/4}\pdisc(G)^{1/2}\geq\frac{\varepsilon \Lambda}{16d^{3/2}},$$
		which reduces to
		$$\pdisc(G)\geq \frac{\varepsilon^2 \Lambda}{1000d^{3/2}}\geq \frac{\varepsilon^2d^{1/4}n}{1000}.$$
		This concludes the proof of Theorem \ref{thm:difficult} again.
	\end{description}
	\end{proof}

   After all these preparations, we are ready to prove the last inequality in Theorem \ref{thm:main}, which we restate as follows.

   \begin{theorem}
       For every $\varepsilon>0$ there exist $c>0$ such that the following holds. Let $G$ be a graph on $n$ vertices of average degree $d$ such that $d\leq n(1/2-\varepsilon)$. Then $$\disc^{+}(G)\geq \frac{cd^{1/4}n}{\log n}.$$
   \end{theorem}

	\begin{proof}[Proof of Theorem \ref{thm:main}]
		We may assume that $\varepsilon\leq 0.1$. Let $c_1,c_2$ be the constants given by Lemma \ref{lemma:maxdeg} with respect to $\delta=0.1\varepsilon$. Let $X$ be the set of vertices of degree more than $(1+\delta)d$, and let $f=e(X)+e(X,X^c)$. If $f\geq \frac{\delta}{10} dn$, we have 
  $$\disc^{+}(G)\geq c_1f-c_2n=\Omega(\delta dn),$$ and we are done. Hence, we may assume $f<\frac{\delta}{10} dn$. Let $G_0$ be the graph we get after removing all edges containing some vertex of $X$. Then $e(G_0)=e(G)-f$, the average degree of $G$ is $ d_0:=(dn-2f)/n\geq (1-\delta)d$, and its density is $p_0=p-f/\binom{n}{2}$. Observe that the maximum degree of $G_0$ is at most $(1+\delta)d<(1+\varepsilon/2)d_0$. Therefore, by Theorem \ref{thm:difficult}, we have	$$\pdisc(G_0)>c_3d_0^{1/4}n$$
with some constant $c_3>0$.  
  
  It only remains to transfer our results obtained for $G_0$ to the original graph $G$. Let $X$ be a matrix such that $X\succeq 0$, $X_{i,i}\leq 1$ for every $i\in [n]$, and attaining $\disc_{G_0}(X)=\pdisc(G_0)$. Write $A_0$ for the adjacency matrix of $G_0$. We have 
  \begin{align*}|\disc_{G_0}(X)-\disc_{G}(X)|&=|\langle X,A_0-A\rangle-(p_0-p)\langle X,J-I\rangle|\\
  &\leq \langle A-A_0,J\rangle+(p-p_0)(n^2-n)=4f.
  \end{align*}
		Thus, if $\pdisc(G_0)\geq 8f$, we have $\pdisc(G)\geq \frac{1}{2}\pdisc(G_0)$. On the other hand, if $8f>\pdisc(G_0)$, then $$\disc^{+}(G)\geq c_1f-c_2n\geq \frac{c_1}{8}\pdisc(G_0)-c_2n.$$
	Hence, combining this with our lower bounds on $\pdisc(G_0)$, and using Lemma \ref{lemma:disc_to_disc}, we get
 $$\disc^{+}(G)\geq \frac{cd^{1/4}n}{\log n}$$
   with some constant $c>0$.
	\end{proof}

	\section{The second eigenvalue}\label{sect:eigenvalue}
	
	In this section, we prove Theorem \ref{thm:eigenvalues}. In Section \ref{sect:hard}, we have seen a number of lower bounds on $\pdisc(G)$ including the positive eigenvalues. Here, we present a simple upper bound in terms of the second eigenvalue of $G$, akin to Lemma \ref{lemma:lambda_2_disc}.  We continue to use the same notation as introduced in Section \ref{sect:easy}, but now we assume that $G$ is $d$-regular, so $\lambda_1=d$ and $v_1=\mathbf{e}$.
	
	\begin{lemma}\label{lemma:lambda_2}
		$\pdisc(G)\leq \lambda_2 n.$
	\end{lemma}
	
	\begin{proof}
 
		Let $X\in \Sym(n)$ be positive semidefinite such that $X_{i,i}\leq 1$ for every $i\in [n]$, and $X$ attains the maximum $\pdisc(G)=\disc(X)$. We have
		$$X\circ A-\frac{d}{n}X=\sum_{i=2}^{n}\lambda_i X\circ (v_iv_i^{T}).$$
		As $X$ and $v_iv_i^{T}$ are positive semidefinite, we also have $X\circ (v_iv_i^{T})\succeq 0$. Therefore, 
		$$\sum_{i=2}^{n}\lambda_i X\circ (v_iv_i^{T})\preceq\lambda_2\sum_{i=2}^{n} X\circ(v_iv_i^{T})\preceq \lambda_2\sum_{i=1}^{n} X\circ(v_iv_i^{T})=\lambda_2 (X\circ I).$$
        Hence,
        $$\disc(X)=\mathds{1}^T\left(X\circ A-\frac{d}{n}X\right)\mathds{1}\leq \lambda_2 \mathds{1}^T(X\circ I)\mathds{1}=\lambda_2\sum_{i=1}^{n}X_{i,i}\leq \lambda_2 n.$$
	\end{proof}

	Lemma \ref{lemma:lambda_2_disc} combined with Theorem \ref{thm:main_sparse} immediately yields the following bounds on $\lambda_2$.
	
	\begin{lemma}
		There exists $c>0$, and for every $\varepsilon>0$ there exists $c_1>0$ such that  the following holds for every sufficiently large $n$:
		$$\lambda_2>\begin{cases} cd^{1/2} &\mbox{if }d\leq n^{2/3},\\
			c_1 n/d  & \mbox{if }n^{2/3}\leq d\leq n^{3/4}.\end{cases}$$
	\end{lemma}
	
	Similarly, Theorem \ref{thm:difficult} implies that $\lambda_2=\Omega(d^{1/4})$ for every $d\leq n(1/2-\varepsilon)$. However, we can get an even stronger bound. In the next theorem, we will sandwich $\Lambda$ between two bounds involving $\lambda_2$, which in turn will imply a lower bound on $\lambda_2$ as well. We highlight that as $G$ is regular, 
    $$\Lambda=-\sum_{i=2}^{n}\lambda_i^3=d^3-T,$$ where $T$ is the number of homomorphic copies of triangles in $G$, i.e.\ six times the number of triangles in $G$.  
	
	\begin{theorem}\label{thm:lambda_2_and_delta}
		There exists $c_1>0$ (depending on $\varepsilon$) such that the following holds. Either $\lambda_2\geq c_1^{-1}d^{1/2}$, or 
		$$c_1dn\lambda_2^2\geq \Lambda\geq \frac{nd^2}{c_1\lambda_2}.$$
	\end{theorem}
	
	\begin{proof}
		By Lemma \ref{lemma:Lambda} and Lemma \ref{lemma:lambda_2}, we may assume that $\Lambda\geq \frac{1}{16}nd^{3/2}$, otherwise $\lambda_2\geq \Omega(d^{1/2})$. But then, 
		$$\lambda_2\geq \frac{\pdisc(G)}{n}\geq \frac{cnd^2}{\Lambda},$$
		where $c>0$ is some absolute constant. This implies the first inequality in our theorem.

		Define the matrix 
		$$Z:=I-\frac{1}{t}A+\frac{d-t}{tn}J,$$
		where $t=\lambda_2$, then $Z\succeq 0$. Hence, taking the Hadamard product of $Z$ with 
            $$Y=\frac{1}{d}\sum_{i=2}^{n}\lambda_i^2 v_iv_i^{T}=\frac{1}{d}(A^2-\frac{d^2}{n}J) \succeq 0 ,$$
            we have $Y\circ Z\succeq 0$. Also, we can calculate (similarly as in Claim \ref{claim:RoZ}) that 
            \begin{align*}
                \disc(Y\circ Z)&=\disc(Y\circ I)-\frac{1}{t}\disc(Y\circ A)+\frac{d-t}{tn}\disc(Y\circ J)\\      
                &=-\frac{d}{n}\cdot \mbox{tr}(Y)-\frac{1}{t}\left(1-\frac{d}{n}\right)\langle Y,A\rangle+\frac{d-t}{tn}\cdot \disc(Y)\\
                &=-d-\frac{n-2d+t}{nt}\cdot \langle Y,A\rangle,
            \end{align*}
where the last equality holds by noting that $\mbox{tr}(Y)=n$ and $\langle Y,J\rangle=0$, which implies that $\disc(Y)=\langle Y,A\rangle$. Since $A=\sum_{i=1}^{n}\lambda_i v_iv_i^{T}$, we have that  $\langle Y,A\rangle=\frac{1}{d}\sum_{i=2}^{n}\lambda_i^3=- \Lambda/d$. Using that $d\leq n(1/2-\varepsilon)$, the coefficient of $ \langle Y,A\rangle$ on the right-hand-side is at least $\frac{\varepsilon}{4t}$, so
		$$\disc(Y\circ Z)\geq \frac{\varepsilon\Lambda}{4dt}-d\geq \frac{\varepsilon\Lambda}{8d\lambda_2}.$$
		As the diagonal entries of $Y\circ Z$ are clearly bounded by 2, we deduce that
		$$\lambda_2\geq \frac{\pdisc(G)}{n}\geq \frac{\disc(Y\circ Z)}{2n}\geq \frac{\varepsilon\Lambda}{16dn\lambda_2}.$$
		This implies the second inequality in our theorem.
	\end{proof}
	
	By comparing the two sides of the inequality in Theorem \ref{thm:lambda_2_and_delta}, we immediately get $\lambda_2=\Omega(d^{1/3})$. 
	This finishes the proof of Theorem \ref{thm:eigenvalues}.
	
	\section{Concluding remarks}
	Say that a graph $G$ is \emph{$\varepsilon$-far} from a graph $H$ on the same number of vertices $n$, if one needs to change at least $\varepsilon n^2$ edges of $G$ to get a graph isomorphic to $H$. Say that $G$ is \emph{$\varepsilon$-close} to $H$ if it is not $\varepsilon$-far.
	
	Our only known examples of $n$ vertex graphs with positive discrepancy $O(n)$ are the Tur\'an graphs $T_r(n)$, and graphs that are $O(1/n)$-close to a Tur\'an graph. This suggest that the positive discrepancy of $G$ might be substantially larger if $G$ is far from being a Tur\'an graph. If $G$ has average degree $d$, the condition $d\leq (1/2-\varepsilon)n$ automatically ensures that $G$ is $\Omega(\varepsilon)$-far, so Theorem \ref{thm:main} can be thought of as a special subcase of this conjecture.  
	
	For any positive integer $t$, de Caen \cite{caen} constructed a $d=2^{4t-2} - 2^{3t-2} - 2^{2t-1} + 2^{t-1}$ regular graph on $n=2^{4t-1}$ vertices with $\lambda_2= 2^{t-1} = \Theta(n^{1/4})$ and $|\lambda_{n}| = 2^{3t-2} - 2^{t-1} = \Theta(n^{3/4})$, and thus having positive discrepancy $O(n^{5/4})$.  It is easy to check that this graph is $\Omega(1)$-far from every Tur\'an graph. Motivated by this example, we propose the following conjecture.
	
	\begin{conjecture}
		Let $\varepsilon>0$, then there exist $c>0$ such that the following holds. If $G$ is an $n$-vertex  graph that is $\varepsilon$-far from every Tur\'an graph (including the empty graph), then
		$$\disc^{+}(G)\geq cn^{5/4}.$$
	\end{conjecture}
	
	It would be already very interesting to find any $\alpha>0$ such that $\disc^{+}(G)\geq cn^{1+\alpha}$ is satisfied in the previous conjecture. Also, one can ask the analogous problem about the second eigenvalue.
	
	\begin{conjecture}
		Let $\varepsilon>0$, then there exist $c>0$ such that the following holds. If $G$ is a regular $n$-vertex graph that is $\varepsilon$-far from every Tur\'an graph (including the empty graph), then its second largest eigenvalue $\lambda_2$ is at least $cn^{1/4}$.
	\end{conjecture}
	
	\section*{Aknowledgements}
	
	We would like to thank Jacques Verstraete and Victor Falgas-Ravry for bringing this problem to our attention, and for sharing their ideas. We also thank Lior Gishboliner for many fruitful discussions, and Igor Balla and Ferdinand Ihringer for their valuable remarks.


\begin{thebibliography}{}
		
		\bibitem{Alon93}
		N. Alon.
		''On the edge-expansion of graphs.''
		Combinatorics, Probability and Computing 11 (1993): 1--10.
		
		\bibitem{AlonMaxCut}
		N. Alon.
		''Bipartite subgraphs.''
		Combinatorica 16 (1996): 301--311.

            \bibitem{AHK99}
            N. Alon, P. Hamburger, A. V. Kostochka.
            ''Regular Honest Graphs, Isoperimetric Numbers, and Bisection of Weighted Graphs.''
            Europ. J. Combinatorics 20 (1999): 469--481.

        \bibitem{AKS05}
        N. Alon, M. Krivelevich, and B. Sudakov.
        ''MaxCut in $H$-free graphs.''
        Combin. Probab. Comput. 14 (2005): 629--647.
		
		\bibitem{AMMN}
		N. Alon, K. Makarychev, Y. Makarychev, and A. Naor. 
		''Quadratic forms on graphs.''
		Inventiones Mathematicae 163 (3) (2006): 499--522.
		
		\bibitem{primes}
		R. C. Baker, G. Harman, and J. Pintz.
		''The difference between consecutive primes, II.''
		Proceedings of the London Mathematical Society, 83 (3) (2001): 532--562.
		
		\bibitem{Balla21}
		I. Balla.
		''Equiangular lines via matrix projection.''
		preprint, arXiv:2110.15842 (2021).
		
		\bibitem{BJS}
		I. Balla, O. Janzer, and B. Sudakov.
		''On MaxCut and the Lov\'asz theta function.''
		Proceedings of the AMS, to appear.
		

        \bibitem{Brouwer}
        A. E. Brouwer. \url{https://www.win.tue.nl/~aeb/graphs/srghub.html}
		
		\bibitem{BCN89}
		A. E. Brouwer, A. M. Cohen, and A. Neumaier.
		''Distance-Regular Graphs.''
		SpringerVerlag, Berlin, 1989.
		
		\bibitem{BS06}
		B. Bollob\'as, and A. D. Scott. 
		''Discrepancy in graphs and hypergraphs.''
		More Sets, Graphs and Numbers: A Salute to Vera Sós and András Hajnal (2006): 33--56.
		
		\bibitem{caen}
		D. de Caen.
		''Large equiangular sets of lines in Euclidean space.''
		Electronic Journal of Combinatorics 7 (2000): \#R55.

        \bibitem{Max-Cut-SDP}
        C. Carlson, A. Kolla, R. Li, N. Mani, B. Sudakov and L. Trevisan.
        ''Lower bounds for Max-Cut in H-free graphs via semidefinite programming. ''
        SIAM J. of Discrete Math. 35 (2021): 1557--1568. 
        
		\bibitem{CW04}
		M. Charikar, and A. Wirth.
		''Maximizing quadratic programs: extending Grothendieck’s Inequalit.''
		FOCS (2004): 54--60.
		
		\bibitem{disc_book}
		B. Chazelle.
		''The Discrepancy Method: Randomness and Complexity.''
		New York: Cambridge University Press (2000).

        \bibitem{eigenvalue_survey}
        D. Cvetković, and P. Rowlinson.
        ''The largest eigenvalue of a graph: A survey.''
        Linear and multilinear algebra 28 (1-2) (1990): 3--33.
		
		\bibitem{DKT14}
		E. R. Van Dam, J. H. Koolen, and H. Tanaka.
		''Distance-regular graphs.''
		preprint, arXiv:1410.6294 (2014).
		
        \bibitem{DHJP}
        J. A. Davis, S. Huczynska, L. Johnson, and J. Polhill.
        ''Partial difference sets with Denniston parameters and cyclotomy.''
        preprint, arxiv:2311.00512 (2023).
  
		\bibitem{Edwards1}
		C. S. Edwards.
		''Some extremal properties of bipartite subgraphs.''
		Canadian J. Math. 25 (1973): 475--485.
		
		\bibitem{Edwards2}
		C. S. Edwards.
		''An improved lower bound for the number of edges in a largest bipartite subgraph.''
		Recent advances in graph theory (Proc. Second Czechoslovak Sympos., Prague, 1974) (1975): 167--181.

        \bibitem{erdos}
        P. Erdős.
        ''Problems and results in graph theory and combinatorial analysis,''
        Graph theory and related topics (Proc. Conf., Univ. Waterloo, Waterloo, 1977), Academic Press, 1979, pp. 153--163.
		
		\bibitem{EGPS}
		P. Erd\H{o}s, M. Goldberg, J. Pach, and J. Spencer.
		''Cutting a graph into two dissimilar halves.''
		J. Graph Theory 12 (1988): 121--131.
		
		\bibitem{ES71}
		P. Erd\H{o}s and J. Spencer.
		''Imbalances in $k$-Colorations.''
		Networks 1 (1971/2): 379--385.
		
		\bibitem{GJS}
		S. Glock, O. Janzer, and B. Sudakov. 
		''New results for MaxCut in H-free graphs.''
		Journal of London Math. Soc.  108 (2023): 441--481.

        \bibitem{Groth}
        A. Grothendieck.
        ''Résumé de la théorie métrique des produits tensoriels topologiques.''
        Bol. Soc. Mat. Sao Paulo, 8 (1953): 1--79.
		
		\bibitem{GW95}
		M. X. Goemans, and D. P. Williamson.
		''Improved approximation algorithms for maximum cut and satisfiability problems using semidefinite programmin.''
		Journal of the ACM, 42 (6) (1995): 1115--1145.
		
		\bibitem{ramanujan}
		S. Hoory, N. Linial, and A. Wigderson.
		''Expander graphs and their applications.''
		Bulletin of the American Mathematical Society, 43 (4) (2006): 439--561.
		
		\bibitem{Ihringer}
		F. Ihringer.
		''Approximately Strongly Regular Graphs.''
		Discrete Math. 346 (3) (2023): 113299.

            \bibitem{KS06}
            M. Krivelevich, B. Sudakov.
            ''Pseudo-random graphs.''
             in: More Sets, Graphs and Numbers, Bolyai Society Mathematical Studies 15, Springer, (2006): 199--262.
		
		\bibitem{LB84}
		J. H. van Lint, and A. E. Brouwer.
		''Strongly regular graphs and partial geometries.''
		Enumeration and design. Academic Press Inc. (1984): 85--122.
		
		\bibitem{alon-boppana}
		A. Nilli.
		''On the second eigenvalue of a graph.''
		Discrete Mathematics, 91 (2) (1991): 207--210.

       \bibitem{V17}
       J. Verstraete.
       ''Pseudorandom Ramsey Graphs.''
       \url{https://mathweb.ucsd.edu/~asuk/notes-pseudoramsey.pdf}
		
	\end{thebibliography}
\end{document}